\documentclass[12pt, leqno]{article}
\usepackage{amsmath,amsfonts,amssymb,wasysym,mathrsfs}
\usepackage[utf8]{inputenc}
\usepackage{shapepar}
\usepackage{graphicx}
\usepackage{amssymb, graphics}
\usepackage{cancel}
\usepackage{hyperref}
\usepackage{ifpdf} 
\usepackage{xcolor}
\newcommand{\R}{{\mathbb R}}
\newcommand{\vp}{\varphi}

\newcommand{\ren}{{\mathbb R}^N}

\newcommand{\be}[1]{\begin{equation}\label{#1}}
\newcommand{\ee}{\end{equation}}

\newcommand{\prf}{\par\smallskip\noindent{\sl Proof. \/}}
\newcommand{\finprf}{\unskip\null\hfill$\;\square$\vskip 0.3cm}

\newenvironment{proof}{\prf}{\finprf}
\newtheorem{theorem}{Theorem}[section]
\newtheorem{lemma}{Lemma}[section]
\newtheorem{corollary}[theorem]{Corollary}
\newtheorem{proposition}[theorem]{Proposition}
\newtheorem{remark}[theorem]{Remark}

\newcommand{\ve}{\varepsilon}

\numberwithin{equation}{section}

\newcommand{\blue}{\color{blue}}

\newcommand{\magenta}{\color{magenta}}

\def\qed{\,\unskip\kern 6pt \penalty 500
\raise -2pt\hbox{\vrule \vbox to8pt{\hrule width 6pt
\vfill\hrule}\vrule}\par}
\definecolor{darkblue}{rgb}{0.05, .05, .65}
\definecolor{darkgreen}{rgb}{0.1, .65, .1}
\definecolor{darkred}{rgb}{0.8,0,0}
\newcommand{\nlc}{\normalcolor}
\begin{document}
\title{\textbf{  Asymptotic behaviour of solutions  and \\ free boundaries of the  anisotropic \\ slow
diffusion equation}\\[7mm]}

\author{\Large  Filomena Feo\footnote{Dipartimento di Ingegneria, Universit\`{a} degli Studi di Napoli
\textquotedblleft Parthenope\textquotedblright, Centro Direzionale Isola C4
80143 Napoli, \    E-mail: {\tt filomena.feo@uniparthenope.it}}  \,, \quad
\Large Juan Luis V\'azquez \footnote{Departamento de Matem\'aticas, Universidad Aut\'onoma de Madrid, 28049 Madrid, Spain. \newline
E-mail:
{\tt juanluis.vazquez@uam.es}} \,,
\\[8pt] \Large  and
 \ Bruno Volzone \footnote{Dipartimento di Matematica, Politecnico di Milano,Piazza Leonardo da Vinci 32, 20133 Milano, Italy. \    E-mail: {\tt bruno.volzone@polimi.it}}}

\date{ \ } 

\maketitle

\begin{abstract}
In this paper we explore the theory of the anisotropic porous medium equation in the slow diffusion range.
After revising the basic theory, we prove the existence of self-similar fundamental solutions (SSFS) of the equation posed in the whole Euclidean space. Each of such solutions  is uniquely determined by its mass. This solution has compact support w.r.t. the space variables. We also obtain the sharp asymptotic behaviour of all finite mass solutions in terms of the family of  self-similar fundamental solutions.  Special attention is paid to the convergence of supports and free boundaries in  relative size, i.e., measured in the appropriate anisotropic way.  The fast diffusion case has been studied in a previous paper by us,  there no free boundaries appear.
\end{abstract}

\setcounter{page}{1}


\medskip


\noindent { \ }

\

\numberwithin{equation}{section}


\section{Introduction}\label{sec.intro}%

We consider the \emph{anisotropic porous medium equation} (APME)
\begin{equation}\label{APM}
u_t=\sum_{i=1}^N(u^{m_i})_{x_i x_i}\quad \mbox{in  } \ \quad Q:=\mathbb{R}^N\times(0,+\infty)
\end{equation}
in dimension $N\geq2$ with exponents $m_i>0$ for $i=1,...,N$.  In case all exponents are the same we recover the well-known equation
$$
u_t=\Delta u^m, \quad m>0\,,
$$
which for $m=1$ is just the classical heat equation. The isotropic cases have been extensively  studied in the literature, see \cite{DK07,  Vlibro, VazSmooth, WZYL}. There is an extensive list of references in the larger field of nonlinear diffusion both in theory and applications, cf.  \cite{DiB93}.
 There are not  many  papers on the anisotropic case of Eq. \eqref{APM}, although it has a clear physical
motivation in  the modelling of  fluid dynamics in anisotropic porous media \cite{Bear72}. Thus, if the conductivities of the media
are different in the different directions, the constants $m_i$ in (1.1) may be different from each other.

We will consider solutions to the Cauchy problem for \eqref{APM} with nonnegative initial data
\begin{equation}\label{IC}
u(x,0)=u_0(x), \quad x\in \R^{N}.
\end{equation}
We will assume that  $u_0\in L^1(\mathbb{R}^N)$, $u_0\ge 0$, and we put $M:=\int_{\mathbb{R}^N} u_0(x)\,dx$, so called total mass. Solutions with finite mass are natural in many physical considerations and will be the ones studied here. We also look for solutions $u\ge 0$.

The presence of the anisotropy produces several difficulties that cannot be approached by classical tools that  have made study of the case of isotropic diffusion so successful. Indeed,  the combination of self-similarity and anisotropy is an uncommon topic in the literature, see \cite{Per2006}. The influence of anisotropy adds a  large amount of  details and consequences to the qualitative and quantitative behaviour of the nonlinear diffusion process.

The mathematical problem we discuss came to our attention due the work by Song et al. in Beijing. They published a number of works on the anisotropic equation mentioned above, both for $m_i>1$ and $m_i<1$. Of interest here are  paper \cite{JS06} where continuous solutions with finite mass are constructed under suitable assumptions the anisotropy of the exponents, and \cite{JS05} where a fundamental solution is constructed for general initial data, i.e., a solution with a Dirac delta as initial data. See also \cite{S01, S01bis} It was supposed to be the basis of asymptotic long-time analysis, but such analysis resisted the passage of time. On the other hand, the questions and boundedness and continuity of the solutions have been considered in more detail in recent years in the literature, c. \cite{H11, BS19}.

In a previous paper \cite{FVV23} we have studied the case of fast diffusion, reflected in the conditions $m_i<1$ for all $i=1,\cdots, N$, under an extra restriction  $\overline{m}>m_c=(N-2)_+/N$ that ensures a unified fast diffusion theory with locally bounded solutions.  Here, $\overline{m}$ is the average of the exponents, $\overline{m}=\frac1{N}{\sum_{i=1}^N m_i}$, the most important parameter in the joint behaviour of the different diffusion terms with different scalings. The new theory keeps a strong similarity to the isotropic fast diffusion equation $u_t=\Delta u^m$ with $ m<1$. A very important property is that nonnegative initial data give rise to a class  of unique weak solutions which are positive for all $x\in\ren$ and $t>0$, and they are continuous.  After that, we were able to contribute the missing analysis of self-similarity for the self-similar fundamental solutions (SSF solutions). Such special solutions  are unique. With this  amount of extra information we were able to  prove the asymptotic behaviour for general finite mass solutions, a main goal of the study.

\medskip

\noindent {\bf Outline of contents and results.} In this paper, we are interested in extending the theory to slow diffusion case where all $m_i>1$,  taking as motivation what is known for the porous medium equation, $u_t=\Delta u^m$ with $ m>1$, that is described in the monograph \cite{Vlibro} and other references. We will therefore consider the problem formed by equation \eqref{APM} with initial data \eqref{IC}.
Our structural assumptions are:

\noindent  (H1) \centerline{$m_i> 1$ \qquad  for all $i=1,...,N$,}

\smallskip

\noindent  (H2) \centerline{$m_i<\overline{m}+2/N $ \qquad  for all $i=1,...,N$,}

where $\overline{m}$ is the average of the exponents, $\overline{m}=\frac1{N}{\sum_{i=1}^N m_i}$.

 The paper begins by a section devoted to recall the preliminaries on self-similarity, renormalization and scalings,  that adapts the concepts introduced in \cite{FVV23} for the fast diffusion case. This is followed by a section on the construction of solutions and the main properties.  We include the material for due reference but the novelties are minor with respect to the fast diffusion case.

 After that we enter into the three main topics: the theory of SSF solutions, the theory of asymptotic behaviour of general  nonnegative solutions with $L^1$ data, and the asymptotic behaviour of the support and free boundary of bounded solutions with compactly supported data. In order to gain a certain perspective of the difficulties and differences with respect to the isotropic case, we make a short comment in the final Section  \ref{sec.comments}.

Turning back to the contents, the qualitative behaviour of slow diffusion is very different from the fast one. The superlinear exponents, $m_i>1$ imply the degeneracy of the equation at the level $u=0$, which gives rise to finite speed of propagation (slow diffusion),  and the appearance of free boundaries. Qualitatively, the solutions have different properties from the fast diffusion case; they are no more positive everywhere and there is limited regularity at the border of the support. This forces the introduction of new tools. In particular, the construction of the unique self-similar fundamental solution (SSF) in Section \ref{sec.exfs} uses a new improved technique that is perfectly suited to deal with finite propagation. The SSF solutions have compact support with respect to  the space variables.  The support properties for the SSF solution $U_M$ and its profile $F_M$ are studied in detail in Section \ref{ssec.support}.

 In Section \ref{sec.asymp} we obtain the asymptotic behaviour of all finite mass solutions in terms of the family of  self-similar fundamental solutions.  Time decay   rates  are derived as well as other  properties of the solutions,  like quantitative boundedness, regularity and rate of expansion of the support with time. The combination of self-similarity and anisotropy is essential in our analysis and creates various mathematical difficulties that are addressed by means of novel methods.

To complete the study,  we establish  in Section  \ref{sec.exr.supp} the sharp large-time behaviour of the support of all nonnegative solutions that evolve from compactly supported initial data, as well as the corresponding asymptotic behaviour of the free boundary that appears as the boundary (i.e.,  the separation hypersurface  that borders the positivity set of $u(\cdot,t)$). We  precisely prove how the support and the free boundary expand in an anisotropic way as $t$ tends to infinity, and we find the exact rates in different coordinate directions.  See precise results in Theorems \ref{exp.cs} (for the supports) and \ref{exp.fb} (written in terms of free boundaries). This is a first contribution to the study of the free boundaries (existence, regularity and large-time behaviour). Such study has been done in the isotropic case with $m>1$ in great detail, see
\cite{BZ72, CaffFdmn80, CVW87, CW90, KKV16, Vlibro},  and remains still a very open task in the anisotropic case.

The final Section  \ref{sec.comments} contains  comments on  related results and open problems. We draw attention here to the comment on  the self-similar distorted balls as stable geometries of the proposed anisotropic evolution. This phenomenon should be worth examining in view of wider applications.

\medskip


\section{Preliminaries}

We will follow the main lines of exposition of the fast diffusion paper \cite{FVV23} with special emphasis on the properties and proofs that strongly distinguish the slow diffusion from the fast diffusion, like the absence of total positivity and the appearance of free boundaries.

This section reviews material on self-similarity that was essentially contained in \cite{FVV23}. 

\subsection{Self-similar solutions and self-similar profiles}\label{ssc.sss}

The self-similar solutions of equation \eqref{APM} are solutions having the form
\begin{equation}\label{sss}
U(x,t)=t^{-\alpha}F(t^{-a_1}x_1,..,t^{-a_N}x_N)
\end{equation}
with suitable constants $\alpha>0$, and $a_1,..,a_n\ge 0$, and a suitable profile $F$. Recalling that we consider only nonnegative solutions that enjoy the mass conservation property, we easily  get that $\alpha=\sum_{i=1}^{N}a_i$. Putting $a_i=\sigma_i \alpha,$ and inserting formula \eqref{sss} into equation \eqref{APM}, we determine in a unique way the values for the exponents \ $\alpha$ and $\sigma_i$ by ensuring that the times variable cancels out. This happens for the precise values (a lucky fact):
\begin{equation}\label{alfa}
\alpha=\frac{N}{N(\overline{m}-1)+2},
\end{equation}
where $\overline{m}=\frac1{N}{\sum_{i=1}^N m_i}$, and
\begin{equation}\label{ai} \ 
 \sigma_i= \frac{1}{N}+ \frac{\overline{m}-m_i}{2}.
\end{equation}

 We still have to find the profile $F$ but let us make a comment first.
Observe that Condition (H1) imposed in the Introduction guarantees that $N/2>\alpha>0$. Then the self-similar solution will decay in time in maximum value like a power of time. This is a typical feature of many diffusion processes. On the other hand, the $\sigma_i $ exponents control the expansion of the solution in the different coordinate directions with time. Condition (H2) on the $m_i$  ensures that $\sigma_i> 0$, i.e the self-similar solution expands as time passes, or at least does not contract, along any of the space coordinate variables. This again a desirable property that we will ensure.

Finally, the \sl profile function \rm $F(y)$   must satisfy the following nonlinear anisotropic stationary equation  \eqref{StatEq} in $\mathbb{R}^N$:
\begin{equation}\label{StatEq}
\sum_{i=1}^{N}\left[(F^{m_i})_{y_iy_i}+\alpha \sigma_i\left( y_i F\right)_{y_i}
\right]=0.\end{equation}
This is a direct consequence of Equation \eqref{APM}.

\begin{proposition}\label{Lem1}
$U(x,t)$ is a self-similar solution to \eqref{APM} of the form \eqref{sss} where $a_i=\alpha \sigma_i$ for all $i=1,\cdots,N$ and $\alpha$ and $\sigma_i$ satisfy \eqref{alfa} and \eqref{ai} if  and only if its profile $F$ satisfies the stationary equation \eqref{StatEq}. Moreover, $\int_{\mathbb{R}^N} U(x,t)\, dx =\int_{\mathbb{R}^N} F(y)\, dy=M$ \ {\rm for all} $t>0.$
\end{proposition}

We recall that in the isotropic case the profile function of the corresponding nonlinear elliptic equation is explicit:
 \begin{equation} \label{Fm}
 F(y;m)=\left( C - \frac{\alpha(m-1)}{2mN}|y|^{{2}}\right)_+^{{{1}}/(m-1)},
 \end{equation}
 with a free constant $C=C(M)>0$ that fixes the total mass $M$ of the solution. This is usually called a Barenblatt profile, \cite{Bar52}.
 As far as we know, there is no comparable explicit solution in the anisotropic equations, a difficulty of the theory. It is easy to show that self-similar solutions of the type  \eqref{sss} are  fundamental solutions to \eqref{APM}, in the sense that they take a Dirac mass as initial trace.

 \begin{lemma}\label{selfsimarefundsolut}
 If $U(x,t)=t^{-\alpha}F(t^{-a_1}x_1,..,t^{-a_N}x_N)$ is the self-similar function  defined in \eqref{sss}, where $a_i=\alpha \sigma_i$ for all $i=1,\cdots,N$, and also $\alpha$ and $\sigma_i$ satisfy \eqref{alfa} and \eqref{ai}, then it is  a fundamental solution of the Cauchy Problem \eqref{APM}-\eqref{IC} if \ $F\ge0$,  $F\in L^1(\ren)$ and it satisfies equation \eqref{StatEq}.
 \end{lemma}

\subsection{Self-similar parabolic variables}

 It will be very useful to zoom the original solution according to the self-similar exponents \eqref{alfa}-\eqref{ai}. The change of variables is
 \begin{equation}\label{NewVariables}
v(y,\tau)=(t+t_0)^\alpha u(x,t),\quad \tau=\log (t+t_0),\quad y_i=x_i(t+t_0)^{-\sigma_i\alpha} \quad i=1,..,N,
\end{equation}
with $\alpha$ and $\sigma_i$ are defined in \eqref{alfa}-\eqref{ai}.
\begin{lemma}\label{Lem1}
If $u(x,t)$ is a solution (resp. supersolution, subsolution) of \eqref{APM}, then $v(y,\tau)$ is a solution (resp. supersolution, subsolution) of
\begin{equation}\label{APMs}
v_\tau=\sum_{i=1}^N\left[(v^{m_i})_{y_i y_i}+\alpha \sigma_i \left(\,y_i\,v\right)_{y_i}\right] \quad \text{ in \quad }\mathbb{R}^N\times(\tau_0,+\infty).
\end{equation}
\end{lemma}
Observe that the rescaled equation does not change with the time-shift $t_0$ (in general $t_0=0$ or $t_0=1$), but the initial value of the new time does, $\tau_0 = \log(t_0)$. In particular if $t_0 = 0$ then $\tau_0 =-\infty$ and the $v$ equation is defined for all $\tau\in \mathbb{R}$.

We stress that this change of variables preserves the $L^1$ norm:
 $$\int_{\mathbb{R}^N}v(y,\tau) \, dy=\int_{\mathbb{R}^N}u(x,t) \, dx\quad \text{ if }\tau=\log (t+t_0)\quad  \forall t\geq t_0.$$


\subsection{Scaling and mass change}\label{sec scaling}

Equation \eqref{APM} is invariant under the following scaling transformation preserving the mass of the solution
\begin{equation}\label{scal.tr}
\widehat u(x,t)=k^{\alpha}\,u(k^{a_1}x_1,\cdots,  k^{a_N}x_N ,kt), \quad k>0,
\end{equation}
assuming that $\alpha (m_i-1)+2a_i=1$ for all $i$ hold.

 It is important to notice that equation \eqref{APM} is also invariant under the scaling transformation that does not alter time:
\begin{equation}\label{scal.tr}
{\mathcal T}_k \,u(x,t)=k u(k^{-\nu_1}x_1,\cdots,  k^{-\nu_N}x_N ,t), \quad k>0,
\end{equation}
that changes solutions into solutions  if  \ $m_i -2\nu_i= 1$ for all $i$, hence $\nu_i=(m_i-1)/2>0$.  The corresponding scaling for the $v$ solution is
\begin{equation}\label{Change mass.v}
 {\mathcal T}_k v(y,\tau)=k v(k^{-\nu_i}y_i,\tau )\,.
\end{equation}
Therefore, the formula for the stationary solutions of \eqref{StatEq} is
\begin{equation}\label{Tk}
{\mathcal T}_k F(y)=F_k(y)=k F(k^{-\nu_i}y_i)\,.
\end{equation}
The stationary equation is invariant under this transformation  if $\nu_i=(m_i-1)/2>0$. Note that this changes the mass ({or the $L^1$ norm})
\begin{equation}\label{Change mass}
\int_{{\mathbb{R}^N}} F_k(y)dy= k\int_{{\mathbb{R}^N}} F(y_i\,k^{-\nu_i})\,dy=k^{\beta} \,\int_{{{\mathbb{R}^N}}} F(z)\,dz,
\end{equation}
with
\begin{equation}\label{beta}
\beta=1+\sum_i\nu_i=1+N(\overline {m}-1)/2=\frac{N}{2}(\overline {m}-m_c)>1.
\end{equation}
Hence, the new mass is $M_k=k^{\beta}M_1$. Same change of mass applies  to ${\mathcal T}_k \,u$  w.r.t. $u$.  This transformation will be used in the sequel to reduce the calculations with self-similar solutions to the case of unit mass.

We refer the reader to the fast diffusion analysis in \cite{FVV23} for proofs and more details related to all this section.

\section{Construction of solutions and main properties}\label{ssec.prelim2}

This section contains a review of technical results regarding the existence of solutions that were also used in \cite{FVV23}.

\subsection{Construction of solutions by approximation}\label{ssec.approx}

The construction is rather standard and it is similar to one presented in \cite{FVV23} for the fast diffusion regime, but we detail some steps that will be useful in what follows.

 Let $u_0\in L^1(\ren)\cap L^\infty(\ren)$ be the nonnegative initial datum. Recalling that $Q=\R^{N}\times (0,+\infty)$, we construct an $L^2$ \sl weak energy solution \rm  $u$,  in the sense that $u\in L^{2}(Q)$, $\frac{\partial}{\partial x_{i}}u^{m_{i}}\in L^{2}(Q)$ \, for all $i=1,\cdots,N$  and it satisfies
\begin{equation}\label{weak.ren}
\int_{0}^{\infty}\int_{\R^{N}}u\varphi_{t}\,dx\,dt=\sum_{i=1}^{N}\int_{0}^{\infty}\int_{\ren} (u^{m_{i}})_{x_{i}}\varphi_{x_{i}}dx\,dt
 -\int_{\ren}u_{0}(x)\varphi(x,0)dx,
\end{equation}
for all the test functions $\varphi\in C_{c}^{1}(\R^{N}\times [0,+\infty))$. These solutions enjoy the following \sl energy estimates\rm\,
\begin{equation}\label{Energywholespace bis}
\begin{split}
4 \sum_{j=1}^N\frac{m_im_j}{(m_i+m_j)^2}\int_0^T\int_{\ren}&\left|\frac{\partial}{\partial x_j}\left(u^{\frac{m_i+m_j}{2}}\right)\right|^2 \, dx\, dt
\\
\leq&\int_{\ren}\left[
\frac{1}{m_i+1}{u_0}^{m_i+1}
\right]\, dx
-
\int_{\ren}\left[
\frac{1}{m_i+1}{u}^{m_i+1}(x,T)
\right]\, dx
\end{split}
\end{equation}
for all $i=1,...,N$ and $T>0$. In particular we obtain
\begin{equation}\label{Energywholespace}
\int_0^T\int_{\ren}
\left|\frac{\partial}{\partial x_i}u^{m_i}\right|^2 \, dx\, dt
\leq\int_{\ren}\left[
\frac{1}{m_i+1}{u_0}^{m_i+1}
\right]\, dx
-
\int_{\ren}\left[
\frac{1}{m_i+1}{u}^{m_i+1}(x,T)
\right]\, dx
\end{equation}
for all $i=1,...,N$ and $T>0$.

\medskip

(i) \textit{Sequence of approximate Cauchy-Dirichlet problems in a ball.}
As a first step we consider a sequence of approximate Cauchy-Dirichlet problems in the ball $B_n(0):=\{x:|x|<n\}$ with an initial datum  ${u_0}_n\ge 0$, that is a suitable approximation of $u_0$ in $B_n(0)$. These Cauchy-Dirichlet problems are not uniformly parabolic at the level $u=0$ because the diffusion coefficients $m_iu^{m_i-1}$ go to zero when $u \to 0$. To overcome this difficulty we construct a sequence of approximate initial data $u_{0,n,\varepsilon}$ which do not take the value $u = 0$ by moving up the initial and boundary datum.
Then we consider the following sequence of approximate problems
\begin{equation}
\label{PCD appr}
\tag{$P_{n,\varepsilon}$}
\left\{
\begin{aligned}
    &(u_{n,\varepsilon})_t=
    \sum_{i=1}^N\left(a_\varepsilon^i(u_{n,\varepsilon})(u_{n,\varepsilon})_{x_i}\right)_{x_i} & \quad \text{in }Q_n, \\
    &u_{n,\varepsilon}(x,0)=u_{0,n,\varepsilon}(x) &\quad  \text{for }|x|\leq n, \\
   &u_{n,\varepsilon}(x,t)=\varepsilon &\quad  \text{for }|x|=n, \ t\geq0,
\end{aligned}
\right.
\end{equation}
where $\varepsilon>0$, $a_\varepsilon^i(z)=m_{i}z^{m_{i}-1}$ for $z\in[\varepsilon, \sup u_{0}+\varepsilon]$ and $u_{0,n,\varepsilon}=u_{0n}+\varepsilon$.

Since problem \eqref{PCD appr} is uniformly parabolic, we can apply  the standard quasilinear theory, see \cite{LSU}, to find a unique solution $u_{n,\varepsilon}(x,t)$, which is bounded from below by $\varepsilon>0$ in view of the Maximum Principle. Moreover, the solutions $u_{n,\varepsilon}$ in this step are $C^\infty(Q_n)$ by bootstrap arguments based on repeated differentiation and interior regularity results for parabolic equations. Using again the Maximum Principle we conclude that
 $\varepsilon\le u_{n,\varepsilon} \le \sup u_{0}+\varepsilon$.

In order to get energy estimates that are uniform in $\ve$ and $n$, we multiply the equation in \eqref{PCD appr} by $\eta_\varepsilon=u_{n,\varepsilon}^q-\varepsilon^{q}$ with $q=m_i$ for some $i$ (see \cite{FVV23} for more details).

\medskip

(ii)\textit{ Passage to the limit as $n\rightarrow\infty$}. We let the ball $B_n(0)$ expand into the whole space for fixed $\ve>0$. The family $\{u_{n,\varepsilon}: n\ge 1\}$ is uniformly bounded in $Q_n$ and also uniformly away from 0.
 We recall that each $u_{n,\varepsilon}$  is a non-negative solution of problem \eqref{PCD appr}. Since $u_{n,\varepsilon}(x,t)\leq u_{n+1,\varepsilon}(x,t)$ on the boundary of the cylinder $Q_n$,  applying the classical comparison principle we get \ $u_{n,\varepsilon}(x,t)\leq u_{n+1,\varepsilon}(x,t)\text{ \ in } Q_{n}$. Thus, we obtain the monotonicity of $u_{n,\varepsilon}$ in $n$ and we are able to pass to the limit as $n\rightarrow\infty$ and we can set (up to extending $u_{n,\varepsilon}(x,t)$ to $\ve$ out of $Q_{n}$)
 $$
u_\ve (x,t):=\lim_{n\to\infty} u_{n,\varepsilon}(x,t).
$$
and we can verify that $u_\ve (x,t)$ verify the Cauchy problem in $\mathbb{R}^N$ with initial datum $u_\varepsilon(x,0)=u_0(x)+\varepsilon$ (see \cite{FVV23} for more details). %

\medskip

(iv) \textit{ Passage to the limit as $\ve \to 0$}. We notice that the family $\{u_{\varepsilon}\}$ is monotone in $\ve$  by the construction of the initial data. Then we define the limit function
\begin{equation}\label{un}
u(x,t)=\lim_{\varepsilon\rightarrow0} u_\varepsilon(x,t)
\end{equation}
as a monotone limit in $ Q$ of bounded non-negative (smooth) functions. Moreover $u$ is an $L^2$ weak energy solution of \eqref{APM} with initial datum $u_0\in L^1(\mathbb{R}^N)\cap L^\infty(\mathbb{R}^N)$ and verifies the energies estimates \eqref{Energywholespace bis}.

\begin{remark}\label{Remark 1}
When $u_0\in L^1(\ren)\cap L^\infty(\ren)$, we may multiply the equation \eqref{PCD appr} by $\phi=u_{n,\varepsilon}^{p}$ with any $p>1$, and integrate by parts to obtain  following energy estimate
\begin{equation}\label{Energy}
\sum_{i=1}^N\int_0^T\int_{\ren}\left|\frac{\partial}{\partial x_i } u_{n,\varepsilon}^{\frac{p+m_j}{2}}\right|^2\,dx\,dt\leq \frac{1}{p+1}\int_{\ren}u_{0,n,\varepsilon,}^{p+1}\,dx.
\end{equation}
Passing to the limit, choosing a particular $i$ and selecting the a value $p=p_i>1$ we get for every $i$
\begin{equation}\label{Energywholespace}
\int_0^T\int_{\ren}\left|\frac{\partial}{\partial x_i } u^{\frac{p_i+m_i}{2}}\right|^2\,dx\,dt\leq C_i\int_{\ren}u_0^{p_i+1}\,dx.
\end{equation}
for a solution $u$ of equation \eqref{APM}.
\end{remark}

\subsection{Solutions with $L^1$ data}
\medskip
Now we recall some properties of solutions constructed in the previous subsection . For the proof we refer the reader to \cite{FVV23}.
Let us first state  the property the monotonicity of $L^p$ norm. Its proof is standard.  

\begin{proposition}\label{decay of the $L^p$}
Let $u$ be the constructed solution with $u_0\in L^{1}(\mathbb{R}^N)\cap L^\infty(\mathbb{R}^N)$. Then, $u(t)\in L^{p}(\mathbb{\R^{N}})$ for all $p\in[1,\infty]$ and
\begin{equation}
\|u(t)\|_{p}\leq \|u_{0}\|_{p}.\label{boundLpnormindata}
\end{equation}
Under conditions (H1) and (H2) we have equality for $p=1$.
\end{proposition}

The next result shows that the set of constructed solutions enjoys the property of $L^1$ contraction in time in the strong form proposed by B\'enilan \cite{Benilan72} as $T$-contraction, a property that implies comparison.

\begin{theorem}\label{contraction}
For every two constructed  solutions $u_1$ and $u_2$ to \eqref{APM} with respective initial data $u_{0,1}$ and $u_{0,2}$ in $L^{1}(\R^{N})\cap L^\infty(\R^{N})$ we have
\begin{equation}\label{L1_contr}
\int_{\R^{N}} (u_1(t)-u_2(t))_+\,dx\le \int_{\R^{N}} (u_{0,1}-u_{0,2})_+\,dx\,.
\end{equation}
In particular, if $u_{0,1}\le u_{0,2}$ for a.e.  $x$, then for every $t>0$ we have $u_1(t)\le u_2(t)$ for all $x\in \R^N$ and $t>0$.
\end{theorem}

\begin{lemma}\label{comparisboundedunbound} Suppose that $u_1$ and $u_2$ are classical  nonnegative solutions to \eqref{APM} defined in a bounded {or unbounded} spatial domain $\Omega$, with smooth boundary, living for a time interval $0\le t\le T$  and suppose that $u_{1}\le u_{2}$ on $\partial   \Omega$. Then the contraction result holds. If  we have \, $\int_{\Omega} (u_{0,1}-u_{0,2})_+\,dx=\infty$ there is no assertion.
\end{lemma}

Finally we state the $L^1$ to $L^\infty$ smoothing effect.

\begin{theorem}\label{L1LI} If $u_0\in L^1(\mathbb{R}^N)\cap L^\infty(\mathbb{R}^N)$, then the constructed solution $u$ to \eqref{APM}-\eqref{IC} under assumptions (H1) and (H2) satisfies
 \begin{equation}\label{Linfty-L1}
 \|u(t)\|_\infty\leq C t^{-\alpha}\|u_0\|_1^{2\alpha/N}\quad \forall t>0,
 \end{equation}
where the exponent $\alpha$ is defined in \eqref{alfa} and $C=C(N,m_1,...,m_N)$.
 \end{theorem}

The proof is a very important fact. It follows from the argument in the Appendix of \cite{FVV23} for the fast diffusion case.

The existence of solutions for datum $u_0\in L^1(\mathbb{R}^N)$ is based on idea to approximate the initial
data by a sequence of bounded integrable functions and then pass to the limit in
the approximate problems. The techniques are rather classical and the key tools needed to pass to the limit are the $L^1$- contraction property and the smoothing effect.

\bigskip
The following Theorem summarizes the existence, uniqueness and all the properties satisfied by the constructed solution with $L^{1}$ data:

\begin{theorem}\label{EUWES} Let the exponents $m_i$ satisfy assumptions (H1) and (H2). Then, for any nonnegative $u_0 \in L^1(\mathbb{R}^N)$ {there is a unique function $u\in C([0,\infty): L^1(\mathbb{R}^N))$ such that $u, u^{m_{i}}\in L^{1}_{loc}(Q)$ for all $i=1,...,N$, and equation \eqref{APM} holds in the distributional  sense in $Q=\mathbb{R}^N\times(0,+\infty)$, with the following additional properties:}

1)  $u(x,t) $ is a uniformly bounded function for each $\tau>0$ and \eqref{Linfty-L1} holds.

2)  Let $Q_\tau=\ren\times (\tau,\infty)$. We have $\partial_i u^{m_{i}}\in L^{2}(Q_\tau)$ for every $i$ and the energy estimates \eqref{Energywholespace} are satisfied.
 Equation \eqref{APM} holds in the weak sense of  \eqref{weak.ren} applied in $Q_\tau$ for every $\tau>0$.

3) Consequently, the maps $S_t: u_0\mapsto u(\cdot,t)$ generate a semigroup of $L^1$ ordered contractions in $L^1_+(\ren)$. The $L^1$-contraction estimates \eqref{L1_contr} are satisfied. The maximum principle applies.

4) Conservation of mass holds: for all $t>0$ we have \ $\int u(x,t)\,dx=\int u_0(x)\,dx$.

5) If we start with initial data $u_0\in L^1(\ren)\cap L^\infty(\ren)$ we may also conclude item 2) with $\tau=0$ and $u(x,t)$ is uniformly bounded and continuous in space and time.
 \end{theorem}

 \noindent {\bf Monotonicity and SSNI property}

Finally we recall a very useful monotonicity property of the solutions. We say that a function $g:\mathbb{R}^N\rightarrow\mathbb{R}$ is \ SSNI if it is a  symmetric function in each variable $x_i$ and a nonincreasing function in $|x_i|$ for all $i$,\textit{ i.e.}
\begin{equation}\label{simm}
 g(x_1,\cdots,x_N)=g(|x_1|,\cdots,|x_N|) \quad \forall x \in \mathbb{R}^N,
\end{equation}
and for all $j=1,\cdots,N$
\begin{equation}\label{mon}
g(|x_1|,\cdots,|x_j| ,\cdots,|x_N|)\leq g(|x_1|,\cdots,|\widehat{x}_j| ,\cdots,|x_N|) \quad \text{ if }|\widehat{x}_j|\leq |x_j|.
\end{equation}
We say that the evolution function $u(x,t)$ is  \ SSNI if it is an \ SSNI function with respect to the space variable for all $t>0$.
The next result states the conservation in time of the SSNI property.

\begin{proposition}\label{Prop 3}
Let $u$ be a  nonnegative solution of the Cauchy problem for \eqref{APM} with nonnegative
initial data $u_0\in L^1(\mathbb{R}^N)$. If $u_0$ is a  symmetric function in each variable $x_i$,  and also a  nonincreasing function in $|x_i|$ for all $i$, then $u(x,t)$  is also symmetric and a nonincreasing function in $|x_i|$ for all $i$ for all fixed $t>0$.
\end{proposition}

 The proof is contained in \cite{FVV23}, but for reader convenience we give some details adding some remarks with respect to the proof proposed in \cite{FVV23}.

\begin{proof}
Let us consider an hyperplane $\mathcal{H}^a_j=\{x_j=a\}$ for any fixed $j\in\{1,\cdots,N\}$ and $a\in\mathbb{R}$. It divides $\mathbb{R}^N$ into the half spaces $\mathcal{H}_j^{a,+}=\{x_j>a\}$ and $\mathcal{H}_j^{a,-}=\{x_j<a\}$. We denote by $\pi_{\mathcal{H}^a_j}$ the specular symmetry that maps a point $x\in \mathcal{H}_j^{a,+}$  into $\pi_{\mathcal{H}^a_j}(x)\in \mathcal{H}_j^{a,-}$, its symmetric image with respect to $\mathcal{H}^a_j$.

Let $u$ be a  nonnegative solution of the Cauchy problem for \eqref{APM} with
nonnegative initial data $u_0\in L^1(\mathbb{R}^N)$. If  for a given hyperplane $\mathcal{H}^a_j$ with  $j=1,\cdots,N$ we have
$$u_0(\pi_{\mathcal{H}^a_j}(x))\leq u_0(x)\, \text{ for all }x\in \mathcal{H}{_j}^{a,+}$$
then for all $t$
\begin{equation} \label{DD}
u(\pi_{\mathcal{H}^a_j}(x),t)\leq u(x,t)\quad \text{ for all }(x,t)\in \mathcal{H}_{j}^{a,+} \times  (0,\infty).
\end{equation}
This result is proved \cite{FVV23} ( see Proposition 4.1) for $a=0$. By translation it is clear by reading the proof that the previous result holds  for any $a\in\mathbb{R}$.
By property \eqref{DD} the solution $u(x,t)$ is a function in $|x_i|$. We want to apply Proposition \ref{comparisboundedunbound} in $H_i^+$, to $u(x,t)$ and to $\widehat{u}(x,t)=u(x_1,\cdots,x_i+h,\cdots,x_N,t)$. We have to check the parabolic boundary conditions. Obviously we have $u(x,0)=u_0(x)\geq u_0(x_1,\cdots,x_i+h,\cdots,x_N)=\widehat{u}(x,0)$.  Now we prove that
\begin{equation}\label{boundary 2}
u(x_1,\cdots,0,\cdots,x_N,t)\geq u(x_1,\cdots,h,\cdots,x_N,t)=\widehat{u}(x_1,\cdots,0,\cdots,x_N,t) \quad \hbox{ for all }t>0.
\end{equation}
Let us consider the  hyperplane  $\mathcal{H}^{\frac{h}{2}}_i=\{x_i=\frac{h}{2}\}$ and the two solutions $u(x,t)$ and $${u}_{1}(x,t)=u(\pi_{\mathcal{H}^{\frac{h}{2}}_i}(x),t)=
u(x_{1},...,h-x_{i},...,x_{N},t)$$ in $\mathcal{H}^{\frac{h}{2},-}_i$. For all $x\in \mathcal{H}^{\frac{h}{2},-}_i $, notice that $\pi_{\mathcal{H}^{\frac{h}{2}}_i}(x)=(x_{1},...,h-x_{i},...,x_{N})$ and since $|h-x_{i}|=h-x_{i}\geq |x_{i}|$, recalling our assumption on $u_0$ we get
$$u(x,0)=u_0(x)\geq u_0(\pi_{\mathcal{H}^{\frac{h}{2}}_i}(x))={u}_{1}(x,0).$$
Moreover by construction
$$u(x,t)={u_{1}}(x,t) \quad \forall x\in\mathcal{H}_i^{\frac{h}{2}} \quad  \forall t>0.$$
Proposition \ref{comparisboundedunbound} applied in $\mathcal{H}_i^{\frac{h}{2},-}$ yields
$$u(x,t)\geq{u}_{1}(x,t) \quad \forall x\in\mathcal{H}_i^{\frac{h}{2},-} \quad  \forall t>0.$$
From this inequality it follows in particular that \eqref{boundary 2} holds taking $x$ such that $x_i=0$. Therefore we can apply Proposition \ref{comparisboundedunbound} to $u,\,\widehat{u}$ in order to obtain
\[
u(x,t)\geq \widehat{u}(x,t),\quad x\in H^{+}_{i},\,t>0.
\]
Now, take any $\bar{x}_{i},\,\bar{\bar{x_{i}}}>0$ such that $|\bar{\bar{x_{i}}}|>|\bar{x_{i}}|$. Set $h:=|\bar{\bar{x_{i}}}|-|\bar{x_{i}}|>0$. Then we have
\[
u(x_{1},...,\bar{x}_{i},...,x_{N})\geq u(x_{1},...,\bar{x}_{i}+h,...,x_{N})=u(x_{1},...,\bar{\bar{x}}_{i},...,x_{N}).
\]
Due to the symmetry of $u$ with respect to each variable, we have that the previous inequality holds for any $\bar{x}_{i},\,\bar{\bar{x_{i}}}$ such that $|\bar{\bar{x_{i}}}|>|\bar{x_{i}}|$.
\end{proof}

There is a weaker version of this property that  applies to nonnegative solutions with compactly supported initial data.

\begin{proposition}\label{Prop 3b}
Let $u$ be a  nonnegative solution of the Cauchy problem for \eqref{APM} with bounded  initial data supported in the parallelepiped  \ $Q(\vec{a})=[-a_1,a_1]\times \cdots \times [-a_N,a_N]$. Then, for every coordinate direction $x_i$ and every fixed  $t>0$, the solution $u(x_1,\cdots, x_N,t)$ is monotone nonincreasing with respect to the variable $x_i$ in the interval $(a_i,\infty)$, when the rest of the variables are fixed. Analogously, the solution is monotone increasing in $x_i$ for negative values $x_i\in (-\infty, -a_i)$.
\end{proposition}

Both propositions are proved by means of the Aleksandrov reflection principle used in \cite{FVV23}, Section 4, for the fast diffusion case.

Combining the results for different directions we obtain a multidimensional result as follows.

\begin{corollary}\label{Cor 3b}Suppose we are in the assumptions of Proposition \ref {Prop 3b}. Let $x_0=(x_{01}...x_{0N})$ be a point in $\ren$ with all coordinates $x_{0i}\ge a_i$ and $K(x_0)$ be the conical region
$$
K(x_0)=\{x: x_i> x_{0i} \ \forall i\}.
$$
Then $u(x,t)$ is monotone nonincreasing along every straight line that starts at $x_0$ and enters $K(x_0)$.
\end{corollary}

This corollary will be used later in the proof of Proposition \ref{conc.as.csdata}.

\medskip

\noindent {\bf Remark.}  We  have to consider  (see Section  \ref{ssec.support}) the case where the positivity set of the fundamental solution $U_1$ with mass one is not $\mathbb{R}^N$, as expected from the compact support property of the construction process. Attention must be paid to this crucial difference with linear diffusion or fast diffusion.


\section{Existence and uniqueness of self-similar fundamental solution.}\label{sec.exfs}
The main result of this Section, which is one of the main results of the paper, regards the existence and uniqueness of self-similar fundamental solution with finite mass. As we will see, for the existence part there is a substantial difference in the fixed point argument used in the proof, which prevents using the crucial barrier function employed in  \cite{FVV23}.


\begin{theorem}\label{fundamental solution} Under the restrictions (H1) and (H2), for any mass $M>0$ there is a unique self-similar fundamental solution $U_M(x,t)\ge 0$ of equation \eqref{APM} with mass $M$ obtained as limit of approximate integrable and nonnegative integrable solutions. The profile $F_M$  of such a solution is an SSNI (separately symmetric and non-increasing) function.  $F_M(y)$ has compact support in all directions. In fact, the set $\Omega=\{y: F_M(y)>0\}$ is open,  bounded and star-shaped around the origin. $F$ is $C^\infty$ smooth inside $\Omega$ and continuous in $\ren$.
\end{theorem}

\subsection{Sketch of the uniqueness part}

The proof of the uniqueness of the  self-similar fundamental  solutions  stated in Theorem \ref{fundamental solution} combines a number of different arguments that we have already established in the fast diffusion paper \cite{FVV23}, also in the fractional $p$-Laplacian. We think they need not be repeated here. Indeed\normalcolor , the SSNI property  of the self-similar solutions can be proved as in \cite[Lemma 6.2]{FVV21}. The star-shaped property  of the set of positivity of profiles $F_M$ follows as in \cite[Lemma 6.3]{FVV21}. Observe that by \cite{H11}, the stationary profile $F_{M}$ is continuous. Finally we refer to \cite[Subsection 6.1]{FVV23} for the mass-difference analysis that yields the uniqueness result.

\medskip

The proof of the existence is divided into several subsections containing steps that contribute needed results.
The first one is a precise control on how large the solutions are and where are they supported.  The analysis of the geometry of the support will be done in Section \ref{ssec.support} using the SSNI monotonicity properties of $F$.

\subsection{A novel upper barrier for the rescaled flow}\label{sec.upper}%

The construction of a set of data  that is conserved by the flow after a certain amount of time is a very novel feature  of our existence proof. As we will see below, it is crucial that ``the flow does not leave the assigned box'' after a certain time.  We argue as follows.

Let $S_\tau$  be the semigroup map associated to the rescaled flow \eqref{APMs},   (\textit{i.e.} the $v$ flow, $S_\tau v_0=v(y,\tau)$).
We take $t_0=1$ in \eqref{NewVariables}, then $\tau_0=0$. We consider bounded  initial data $\vp\ge 0$ with the following  conditions in terms of constants $M,L,R>0$ to be chosen:

-A1)  $\int \vp(x)\,dx=M$.

-A2) $0\le \vp(x)\le L$.

-A3) $\vp(x)=0$  if $x$ does not lie in the box: $Q(R)=\{x: |x_i| \le R  \quad \forall i=1,\ldots, N\}$.

-A4)  $\vp$ is SSNI.

Let  us call $\mathcal{K}$ the set of functions $\vp\ge 0$ satisfying A1), A2), A3),  A4).  We recall that $\mathcal{K}=\mathcal{K}(M,L,R)$. It is clear that $\mathcal{K}$ is a closed and convex subset of $L^1(\ren)$.

Next, we observe that under the condition $M \le 2^{N} L\,R^N$ the set $\mathcal{K}$ is not empty. Indeed, let us take the radius $R_1$ such that $M = 2^{N}L\,R_1^N$, then we have $0<R_1<R$.
Let now $\vp(x)=L$ if $x\in Q(R_1)$, and $\vp(x)=0$ otherwise. We have
$\int \vp(x)\,dx=M$ and $\vp(x)\in \mathcal{K}$ .

\begin{proposition}  For every $\tau_1>0$ and for every $M>0$,  there is a choice of $L$ and $R$ such that, under the above conditions A1)-A4), the flow map $S_{\tau_1}$  satisfies
$$
S_{\tau_1}(\mathcal{K})\subset \mathcal{K}.
$$
\end{proposition}

{\sc Proof.} We can put $M=1$ by rescaling (see Section \ref{sec scaling}).

(1) The conservation of mass is true.

(2) By Proposition \eqref{Prop 3}, if the initial datum $v_0(y)$ is  SSNI, then $v(y,\tau)$ is also SSNI.

(3)  To prove A2 we argue as follows. Let us consider $\tau_1=\log(t_1+1)$ with $t_1>0$.
The universal $L^1$-$L^\infty$  estimate \eqref{Linfty-L1} for all solutions implies as we know that
$$
v(y,\tau)\le (t+1)^\alpha u(x,t)\le C(t+1)^\alpha t^{-\alpha} M^{\gamma}=C M^{\gamma}(1+(1/t))^\alpha=
C M^{\gamma}\left(1+\frac{1}{e^\tau-1}\right)^\alpha
$$
Then for $\tau\geq\tau_1$ we get
$$
v(y,\tau)\le C M^{\gamma}\left(1+\frac{1}{e^{\tau_1}-1}\right)^\alpha.$$
If $L$ is large enough (\textit{i.e.} $L\geq C M^{\gamma}(1+\frac{1}{e^{\tau_1}-1})^\alpha$) we get
$$
v(y,\tau)\le L \text{ for } \tau\geq\tau_1.
$$

(4) Now we begin to check the property of not leaving the box $Q(R)$ if $R$ is well chosen once $M$ and $L$ are fixed. The novelty of the  argument we present here lies in the control of the expansion of the support. Namely, we check the property of not leaving the box in any of the horizontal space directions.

(i) We prefer to go back to the $u$ variable in the time interval $0\le t\le  t_1$
and we consider the domain $D_1=\{x: x_1\ge R\}$ that is a half-space along the $x_1$ axis starting at $x_1=R$.
 We will compare $u$ in $D_1$ for times $0<t< t_1$ with a super-solution $\overline{u}$ that we choose as a one dimensional super-solution of the PME in that direction $e_1$.

(ii) {\bf The barrier}. Indeed, we can take as $\overline{u}(x,t)$ the explicit 1D travelling wave with speed $A<1$ of the form
$$
\overline{u}^{m_1-1}(x,t)= c(m_1)A(At+K-x_1)_+,
$$
where $c=c(m_1)$ is a known constant ($c(m_1)\leq\frac{m_1-1}{m_1}$) where $A$ is fixed and $K>0$ will  be chosen soon.

We check that we can apply the parabolic comparison of Lemma \ref{comparisboundedunbound}. Recall that we start from an initial datum $u_0(x)=v_0(x)$ belonging to $\mathcal{K}$. Using A3), the initial datum $u_0$ is zero in $D_1$, then $u(x,0)\leq\overline{u}(x,0)$ for $x\in D_1$.
Moreover, by A2) we know that $u(x,t)\le L$ for all $x$ and $t$ and in particular $u(x,t)\le L$ on the lateral boundary of $D_1$ for all $t>0$. Then at $x_1=R$ we have to get
$$
L^{m_1-1}\leq c(m_1)A(At+K-R) \quad \mbox{for all } \ t>0,
$$
 This is possible by a convenient choice of $K$, for instance $cA(K-R)= L^{m_1-1}$, thus
$$
K= R + c^{-1}A^{-1}L^{m_1-1}.
$$

In conclusion,  we can apply the parabolic comparison. In this way we conclude that
$$
u(x,t)\le \overline{u}(x,t) \quad \mbox{ if } \ x_1\ge R, \ t>0.
$$
This in particular means that
$$
u(x,t)=0 \quad \mbox{ if } \ x_1\ge R+  c^{-1}A^{-1} L^{m_1-1}+At\,.
$$

(iii) We translate this in terms of the rescaled variables $v,y,\tau$. We get
$$
v(y,\tau)=0  \quad \mbox{ if } \ y_1\ge (R+ c^{ -1 }A^{-1}L^{m_1-1}+At)(t+1)^{-\sigma_i\alpha}
$$
We are interested in $t=t_1$ (i.e. $\tau=\tau_1$) where
$$
v(y,\tau_1)=0  \quad \mbox{ if } \ y_1\ge (R+ c^{ -1 }A^{-1}L^{m_1-1}+At_1)(t_1+1)^{-\sigma_i\alpha}
$$
The crucial thing is that when  $R$ is large enough we get
$$
(R+ c^{-1}A^{-1}L^{m_1-1} + t_1 )\,( t_1+1 )^{-\sigma_i\alpha}\le R.
$$
Then we get $v(y,\tau_1)=0$ for $y_1\ge R$. This is what we wanted to prove.

The argument to control the support in the other spatial directions of the box is the same.
This ends the proof of the Proposition.
\qed

\medskip

{\bf Remarks. }  We observe that the technique of the explicit supersolution with a given tail used in the fast diffusion papers \cite{FVV21} and  \cite{FVV23} has not been used this time. The argument here relies on comparison with
one directional travelling waves with compact support, plus a careful inspection of the expansion rates.

The idea of comparing with one dimensional supersolutions can be used for fast diffusion in order to do a proof in this style. This idea is exploited in the paper  \cite {Vaz24} where comparison with VSS solutions is  very accurate. We do not need a very strict accuracy in the present situation.

\subsection{Existence of a fixed point}

After the introductory step of the previous section we prove that the flow map has a fixed point, in other words, a periodic solution. We need the following result compactness result.

\begin{lemma}\label{lemma.precomp}
The image set \ $Y=S_{\tau_1}(\mathcal K(L_1))$  is relatively compact in $X=L^1(\ren)$.
\end{lemma}

\begin{proof} The proof presents some difference w.r. to the one of \cite[Lemma 6.5]{FVV23} in the fast diffusion case. Indeed, here we have to take more care in the use of energy estimates due to porous medium range $m_{i}>1$. First of all, the fact that the  image set $S_h(\mathcal K)$ is bounded in $L^1(\ren)$ and $L^\infty(\ren)$ by already established estimates using the definition of $v$ in terms of $u$. We use then the Fr\'echet-Kolmogorov theorem, that says that a subset $Y\subset L^1(\ren)$ is relatively compact in $L^1(\ren)$ if and only if the following two conditions hold

(FK1) (Equi-continuity in $L^1$ norm)
$$
\displaystyle \lim_{| z | \to 0} \int_{\ren} |f(y)-f (y+z)|\,dy = 0
 $$
 and the limit is uniform on $f\in Y$.

(FK2) (Equi-tightness) We must have
$$
\displaystyle \lim _{r\to \infty }\int _{|x|>r}\left|f(y)\right|\,dy=0
$$
and the limit is uniform  on $f\in Y$.

In our case property  (FK2) holds since  the functions in $Y$ have a  uniformly  bounded support.

For the proof of (FK1) we proceed as follows. Let $v(\cdot,\tau)= S_{\tau}\phi$ and let us go back to the original $u$ formulation to retrieve some convenient energy inequalities.

Putting then $p_i=2q-m_i$ in \eqref{Energywholespace} with some large $q$ (the same for all $i$) we conclude that all the derivatives $\partial_i u^{q}$ are bounded in $L^2(0,T:L^2(\ren))$ for $T>0$ since the integrals in the right-hand side are uniformly bounded. It easy follows that the same is true for $\partial_i v^{q}$.

We may continue in a more traditional way \cite{FVV23}.  The above energy estimate  means that for some $\widetilde{\tau} \in (h/2,h)$ the integral
$$
\int_{\ren} \left|\frac{\partial}{\partial x_i } v^{q}(y,\widetilde{\tau})\right|^2\,dy \le \frac2{h}\int_{h/2}^h\int_{\ren} \left|\frac{\partial v^{q}(y,{\tau})}{\partial x_i }\right|^2\,dy d\tau \le C_2/h,
$$
where $C_2$ depends on $L$, $m_i$, $q$, the mass $M$ and $h$.
By an easy functional immersion this implies that  for every small displacement  $z$ with $|z|\le \delta$ we have for every $r>0$
$$
\int_{B_r(0)} |v^{q}(y,\widetilde{\tau})-v^{q}(y+z,\widetilde{\tau})|\,dy \le \delta C_3
$$
and $C_3$ is a constant that depends only $r$, $h$ and on $C_2$. Moreover, using H\"{o}lder inequality and the inequality $|a-b|^p\leq|a^p-b^p|$ for $a,b>0$ and $p>1$ we get
$$
\int_{B_r(0)} |v(y,\widetilde{\tau})-v(y+z,\widetilde{\tau})|\,dy \le \delta^{1/q} C_4
$$
and $C_4$ is a constant that depends only $r$, $h$, $N$, $p_i$ and on $C_2$. This equi-continuity bound in the interior is independent of the particular initial data in $\phi  \in \mathcal K$.
Putting $r>R$ and using  A3) we get full equi-continuity at $\tau=\widetilde{\tau}$:
$$
\int_{\ren} |v(y,\widetilde{\tau})-v(y+z,\widetilde{\tau})|\,dy \le \ve
$$
uniformly on $\phi\in \mathcal K$ if $\delta$ is small enough. Since both $v(y,\tau)$ and $v(y+z,\tau)$ are solutions of the renormalized equation, we conclude from the $L^1$ contraction property \eqref{contraction} that
$$
\int_{\ren} |v(y,\tau)-v(y+z,\tau)|\,dy \le \ve
$$
uniformly on $\phi\in \mathcal K$ for all $\tau\ge \widetilde{\tau}$, in particular for $\tau=\tau_1$.
This makes the set $S_{\tau_1}(\mathcal K)$ precompact in $L^1(\ren)$.
\end{proof}

It now follows from the Schauder Fixed Point Theorem,  see \cite{Evansbk}, Section 9,
that there exists at least  a fixed point $\phi_{\tau_1} \in \mathcal{K}$, \textit{i.\,e.,} we have  $ S_{\tau_1}(\phi_{\tau_1}) = \phi_{\tau_1}$. The fixed point is in $\mathcal{K}$, so it is not trivial because  its mass is $M$.

Iterating the equality we get periodicity for the orbit $V\nlc_{\tau_1}(y, \tau)$  starting at $\tau = 0$
from $V_{\tau_1}(y,0)=\phi_{\tau_1}(y)$:
$$
V_{\tau_1}(y,\tau+ k\tau_1) =  V_{\tau_1}(y,\tau )\quad  \forall \tau > 0,
$$
This is valid for all integers $k\geq1$. It is not a trivial orbit, $ V_{\tau_1}\not\equiv 0$.
The next result examines the role of periodic solutions.

\medskip

\subsection{The fixed point is stationary}
The proof of the next lemma is similar to the arguments in \cite[Lemma 6.6]{FVV23}, but we give the details for reader's convenience  and to avoid confusions since the argument is long and delicate.

\begin{lemma}\label{lemma.periodicmeans}
Any periodic solution  of our renormalized problem, like $ V_{\tau_1}$, must be stationary in time. We will write $V_{\tau_1}(y,\tau)=F(y)$.
\end{lemma}

\begin{proof} The proof follows the ideas of the proof of uniqueness in \cite[Subsection 6.1]{FVV23} . Thus, if $ V\nlc_1$ is periodic solution that is not stationary, then $ V\nlc_2(y,\tau)= V\nlc_1(y,\tau+ c)$ must be different from $ V\nlc_1$ for some $c>0$, and both have the same mass. With notations as above we consider the functional
 \begin{equation*}
 J[ V_1, V_2](\tau)=\int_{\mathbb{R}^N} ( V_1(x,\tau)- V_2(x,\tau))_+\,dx.
 \end{equation*}
 By the $L^1$-accretivity of the operator  this is a Lyapunov functional, \textit{i.e.}, it is nonnegative and nonincreasing in time.
 By the periodicity of $ V\nlc_1$ and $ V\nlc_2$, this functional must be periodic in time. Combining those properties we conclude that it is constant, say $C\ge 0$, and we have to decide whether $C$ is a positive constant or zero.  In case $C=0$ we conclude that $V_1=V_2$ and we are done.

We want eliminate the other option, $C>0$. We will prove that for two different solutions with the same mass this functional must be strictly decreasing in time. The main point is that such different solutions with the same mass must intersect. We define at a certain time, say $\tau=0$, the maximum of the two profiles $ V\nlc^*(y,0)=\max\{ V\nlc_1(y,0), V\nlc_2(y,0)\}$, and the minimum $ V\nlc_*(y,0)=\min\{ V\nlc_1(y,0), V\nlc_2(y,0)\}$. Let $ V\nlc^*(y,\tau)$ and $ V\nlc_*(y,\tau)$ the corresponding solutions for $\tau>0$. We have for every such $\tau>0$
  \begin{equation}\label{eq.order2}
 V\nlc_*(y,\tau)\le  V\nlc_1(y,\tau),  V\nlc_2(y,\tau)  \le  V\nlc^*(y,\tau).
\end{equation}
On the other hand, it easy to see by the definitions of $ V\nlc^*(0)$, $ V\nlc_*(0)$ that
$$
\int_{\mathbb{R}^N}  V\nlc^*(y,0)\,dy=M+ J[ V\nlc_1, V\nlc_2](0), \quad  \int_{{\mathbb{R}^N}}  V\nlc_*(y,0)\,dy=M- J[ V\nlc_1, V\nlc_2](0).
$$
Since $V\nlc^*(y,0)$ and $ V\nlc_*(y,0)$ are ordered, this difference of mass is conserved in time: for $\tau>0$
\begin{equation}
\int_{{\mathbb{R}^N}} ( V\nlc^*(y,\tau)- V\nlc_*(y,\tau))\,dy= 2J[ V\nlc_1, V\nlc_2](0).\label{eqJ2}
\end{equation}
Now, since $ V\nlc_{1},\, V\nlc_{2}$ have the same mass, \eqref{eq.order2} and \eqref{eqJ2} imply that for $\tau>0$
$$
\int_{\mathbb{R}^N} ( V\nlc_2(y,\tau)- V\nlc_1(y,\tau))_+\,dy=\int_{{\mathbb{R}^N}} ( V\nlc_1(y,\tau)- V\nlc_2(y,\tau))_+\,dy\le J[ V\nlc_1, V\nlc_2](0),
$$
but the constancy of  $J[ V\nlc_1, V\nlc_2]$ forces to have an equality, occurring only if the solution $ V\nlc^*(y,\tau)$ equals the maximum of the two solutions $ V\nlc_1$ and $ V\nlc_2$, and the solution $ V\nlc_*(y,\tau)$  equals the minimum of the two solutions.

\noindent $\bullet$ \textit{ Strong Maximum Principle.\nlc} Finally, we use the strong maximum principle (SMP for short) to show that the last conclusion is impossible in our setting.

i) We stress that $V_1,V_2$ and $V^*$ are periodic SSNI functions of some period $T_0$. Let us assume there at $y_0=0$ we have the option \ $V^*(0,T_0)=V_1(0,T_0)$. Since $V^*$ is SSNI and has positive mass it follows that $V^*(0,T_0)>0$.  By continuity (see Theorem 1 of \cite{H}) $V_1(y,\tau),V_2(y,\tau)>c>0$  in a neighbourhood $I(0)$ of $0$ for all $\tau\in(T_0-\delta, T_0+\delta)$ for a suitable $\delta>0$. Then we can prove locally smoothness for them because the equation is not degenerate (see  \cite[ Theorem 6.1, Chapter V]{LSU}) and as a consequence we can apply the evolution maximum principle (see \cite{Ni, PS2007}) in $ I(0)\times (T_0-\delta, T_0+\delta)$ for a suitable $\delta>0$. By our assumption $V^*(0,T_0)=V_1(0,T_0)$ and we known that $V^*(y,\tau)\geq V_1(y,\tau)$ in $I(0)\times (T_0-\delta, T_0+\delta)$, then  SMP implies $V^*(y,\tau)=V_1(y,\tau)$ in $I(0)\times (T_0-\delta, T_0+\delta)$.

ii) Suppose $y_0\ne 0$ and that $V^*(y_0,T_0)=V_1(y_0,T_0)>0$. We connect the point $y_0$ with $y=0$ by a segment. Since $V^*$ is SSNI we have $V^*(y,T_0),V_1(y,T_0)>0$ on this segment and its symmetric segment of endpoint $-y_0$ and $0$. We denote by $\gamma$ the union of these two segments and by continuity it follows that $V^*(y,\tau),V_1(y,\tau)>0$ for $y\in\gamma_\varepsilon:=\{y+\varepsilon \textbf{e}: y\in \gamma,\textbf{e}\in {\mathbb{S}}^{N-1}\}$ and $\tau\in(T_0-\delta, T_0+\delta)$ for some suitable $\varepsilon,\delta>0$. Then we can prove locally smoothness for them because the equation is not degenerate (see  \cite[ Theorem 6.1, Chapter V]{LSU}) and as a consequence we can apply the evolution maximum principle (see \cite{Ni, PS2007, Vsmp}) in $ \gamma_\varepsilon\times (T_0-\delta, T_0+\delta)$. By our assumption $V^*(y_0,T_0)=V_1(y_0,T_0)$ and we know that  $V^*(y,\tau)\geq V_1(y,\tau)$ in $\gamma_\varepsilon\times (T_0-\delta, T_0+\delta)$, then  SMP implies $V^*(y,\tau)=V_1(y,\tau)$ in $\gamma_\varepsilon\times (T_0-\delta, T_0+\delta)$.

iii) We stress that by the SSNI property the set $\Omega_1=\{y: V_1(y,T_0)>0\}$ is star-shaped set from the origin, \textit{ i.e.} for every $\bar y\in\Omega_1$ the segment from $y_0$ to $\bar y$, $y=\bar y+s(y_0-\bar y)$ with $s\in(0,1)$, belongs to $\Omega_1$. By this property of $\Omega_1$ and our previous argument in ii) we conclude that $V^*(y,T_0)=V_1(y,T_0)\geq V_2(y,T_0)$ for $y\in\Omega_1$.

iv) We have that $V_1(y,T_0)\ge V_2(y,T_0)$ on the boundary of the star-shaped set $\Omega_1$. Observing that  $V_1(y,T_0)=0$ on $\partial \Omega_1$, we must have  $V_2(y,T_0)=0$ on $\partial \Omega_1$. Using now the SSNI property of $V_2$ we get $V_2(y,T_0)=0$ outside on $\Omega_1$. Then $V_1\geq V_2$  in $\mathbb{R}^N$ and we conclude as in the previous case.
\end{proof}

{\bf Conclusion.}
A choice of profile $F(y)=V_{\tau_1}(y,\tau)$ provided by Lemma \ref{lemma.periodicmeans} and the corresponding solution
\[
U(x,t)=t^{-\alpha}F(t^{-a_1}x_1,..,t^{-a_N}x_N)
\]
to equation \eqref{APM} gives the existence of a self-similar profile of mass 1. If we wish to recover the result for any mass $M>0$, it is enough to use the scaling transformations described in Subsection \ref{sec scaling}. Finally, the continuity of the solutions like $F$  comes from \cite{H11}.
The $C^{\infty}$ regularity of $F$ inside the positivity set can be recovered from the  continuity itself. Indeed, $y$ is in the positivity set $\Omega(F)$, take a constant $c>0$ such that $F\geq c$ in a neighbourhood of $y$. Then there is locally an upper and lower bound for $F$, thus we can apply the classical quasilinear regularity theory. Hence, the proof of Theorem \ref{fundamental solution} is completed.

\subsection{Property of monotonicity with respect to the mass}\label{ssec.monotM}

The following monotonicity property of the self-similar profile $F_{M}$ with respect to the total mass  follows more directly than in the FDE case. It is an essential property used in the asymptotic behavior described in Section \ref{sec.asymp}.
\begin{proposition}\label{prop.mon.M} The profile $F_M$ is monotone increasing with respect the mass $M$: if $0<M_{1}<M_{2}$, then $F_{M_{1}}(y)\leq F_{M_{2}}(y)\quad \forall y\in \mathbb{R}^N$.
\end{proposition}

\noindent {\sl Proof.}
 Let us suppose $M_2>M_1>0$. 
By uniqueness of the profile of every mass (see Theorem \ref{fundamental solution}) and \eqref{Tk}, we have
$$
F_{M_2}(y)=k F_{M_1}(k^{-(m_1-1))/2}y_1,\cdots,k^{(m_N-1)/2}y_N)
$$
where $k>1$ is such that $M_2=M_1k^{\frac{N}{2}(\bar m-m_c)}$. Then $F_{M_2}(0)>F_{M_1}(0)$. Moreover, by the monotonicity properties of $F_M$ we also deduce from $k^{-(m_i-1))/2}<1$ that
$$
F_{M_1}(k^{-(m_1-1))/2}y_1,\cdots,k^{(m_N-1)/2}y_N)\ge F_{M_1}(y_1,\cdots,y_N).
$$
It follows that $F_{M_2}(y)\ge F_{M_1}(y)$ for all $y$.

\section{Support properties for $F_M$ and $U_M $}\label{ssec.support}

 We have proved that the self-similar profile $ F_M(y) $ has compact support in $\ren$, hence the fundamental solution $ U_M(x,t) $  has compact support for every fixed $t>0$. This is in complete contrast with the anisotropic diffusion equation in the fast regime (i.e., all $m_i$ less than 1) where the fundamental solution is everywhere positive, see \cite{FVV23}. This contrast was already well known in the isotropic case as the dichotomy between finite propagation and infinite speed of propagation.

Let us analyze a bit more precisely the property of finite propagation for the fundamental profile. Let $F_1(y)$ be the profile of mass $1$. By Theorem \ref{fundamental solution} we know that it is a continuous and nonnegative function in $\ren$, we also know that its positivity set is a bounded open subset $\Omega(F_1)$ of $\ren$. Finally,  $0\in \Omega(F_1)$, $\Omega(F_1)$ is connected and star-shaped around the origin, and $F_1$ is $C^{\infty}$ inside $\Omega(F_1)$. The closure of $\Omega(F_1)$ is called the support of $F_1$.

\subsection{Description of the support of $F_1$}

We proceed as follows. For any unit direction $\bf e\in \mathbb{S}^{N-1}$ we define the maximal length of $\Omega(F_1)$ in that direction
\begin{equation}\label{R1}
R_1({\bf e})=\sup\{r>0: \ r{\bf e}\in \Omega(F_1)\}.
\end{equation}
We already know that $R_1({\bf e})$ is a bounded function defined in $\mathbb{S}^{N-1}$ and it is also uniformly bounded away from 0.

\begin{proposition} \label{prop.R(e)} The  function  $R_1({\bf e})$ is positive and continuous on \ $\mathbb{S}^{N-1}$. It is Lipschitz continuous away from the coordinate frame (i.e. the set of points with at least one coordinate equal to zero)\normalcolor. Moreover, the set $\Gamma(F_1)$ described in polar coordinates as
\begin{equation}\label{Gamma}
\Gamma(F_1)=\{x=(r,{\bf e}): \  {\bf e}\in \mathbb{S}^{N-1}, \quad  r= R_1({\bf e}) \}
\end{equation}
is a continuous hypersurface, which coincides with the boundary of $\Omega(F_1)$ in $\ren$, an object called the \sl free boundary of $F_1$. The support of $F_1$, $S(F_1)$, is the disjoint union of $\Omega(F_1)$ and $\Gamma(F_1)$.
\end{proposition}

\noindent {\sl Proof.} (1) The proof is easy for the points $x\in \Gamma(F_1)$ with all coordinates different from zero. By symmetry we may assume that $x$ has only positive  coordinates, $x=(x_i)$ with $x_i>0$  for all $i$. We draw the line ${\bf L}(x)$ passing through $0$ and $x$. We known that for points $\lambda \, x$ with $\lambda\ge 1$ we have $F_1(\lambda\, x)=0$.
{We point out that by the SSNI monotonicity properties of $F_1$, the cone}
$$
K_*(x)=\{y\in \ren : \ y_i\ge x_i \quad \forall i\}.
$$
{is completely contained in $\R^{N}\setminus \Omega(F_{1})$, since $F_{1}$ is monotone nonincreasing along every straight line that starts at $x$ and enters $K_*(x)$. Thus $F_1(y)=0$ for all $y\in K_*(x)$.}
A simple inspection of the geometry shows that there exists an angle $\theta(x)>0$ such that $F_1(y)=0$ in a half cone  with vertex at $x$, directrix line $\bf L$ and angle $\theta(x)$, and pointing forward, let us call it $K_+(x,{\bf L}(x), \theta(x))$. The inside of this cone is forbidden for all points of $\Gamma(F_1)$ (by definition of this set as a limits  of points where $F_1$  has  positive values). Next we consider the opposed region
$$
K_-(x)=\{y\in \ren : \ x_i/2 < y_i< x_i \quad \forall i\}
$$
and conclude that $F_1(y)>0$ inside  that region, so that there exists an angle $\theta'(x)>0$  such that $F_1(y)>0$ in a partial  cone $K_-(x,\bf L, \theta'(x))$  with vertex at $x$, directrix line ${\bf L}(x)$ and angle $\theta'(x)$, and pointing backwards. This region is necessarily contained in $\Omega(F_1)$, so it is forbidden for points of $\Gamma(F_1)$.

The combination of both bounds implies that for points of $\Gamma(F_1)$ near $x$, of the form $y=(r',{\bf e'})$  with $r'=R_1({\bf e'})$ and ${\bf e'}$ close to ${\bf e}$, we have
\begin{equation}\label{LIps}
|r'-r|\le C|{\bf e'}-{\bf e}|
\end{equation}
and the constant $C$ depends on the previous angles  $\theta(x), \theta'(x)>0$ that in turn depend continuously on the base point $x$.

(2) The rest of points $x$ with at least one coordinate equal to zero is what we call the coordinate frame. It includes in particular all the coordinate axes and also $k$-planes spanned by several axes. On that  set we can still find a positive angle in the outgoing direction, which produces an upper control for $ |r'-r|$ like in \eqref{LIps}, but not the corresponding lower bound. However, we can at least  prove that at such points $R_1({\bf e})$ is at least continuous by a simpler argument using the continuity and monotonicity of $F_1$.

(3) It is then immediate to see that $\Gamma(F_1)$  is the topological boundary of $\Omega(F_1)$, i.e., the free boundary (of course, $\Gamma(F_1)$  is also the  boundary of the support $S(F_1)$). We have shown that it can be parametrized in polar coordinates in a continuous way.  \qed

\begin{corollary} \label{coroll.OF} There are positive constants $c_1,\cdots, c_N$ depending only on the parameters of the problem such that $\Omega(F_1)$ is contained in the box
$$
\Omega(F_1) \subset [-c_1,c_1] \times  \cdots    \times [-c_N,c_N].
$$
\end{corollary}

 In the previous Corollary  the optimal constants  are attained on the coordinate axes by the monotonicity properties of $F_1$.


\subsection{Expansions }

We will see in Section \ref{sec.exr.supp} that the property of finite support for the fundamental solution can be extended to every solution with an initial datum that is nonnegative, bounded and compactly supported. This is what we call finite propagation property generally speaking.
In that section we will need some extra details about modifications of $\Omega(F_1)$ and $\Gamma(F_1)$ that we present next.
The problem comes from the fact that is the anisotropic case $\Omega(F_1)$ and its modifications will be a kind of ``distorted balls''
and the different elongations have to be taken into account.

(1) We first consider the linear expansion of $\Omega(F_1)$ with parameter $\lambda>0$ defined as
$$
E_{\lambda}(\Omega(F_1))=\{z=\lambda y, \ y\in   \Omega(F_1) \}.
$$
In the same way we define $E_{\lambda} (\Gamma(F_1))$. For $\lambda>1$, since $F_{1}$ is SSNI we have $ \Omega(F_1) \subset E_{\lambda}(\Omega(F_1))$, with strict inclusion. Besides we have $E_{\lambda}(\Gamma(F_1))$ encloses $\Gamma(F_1)$ with strict separation.

\medskip

(2) Let us compare the supports of the profiles $F_M$ for different masses, say $M=M_1$  and $M=1$.
We get 
$$
\Omega(F_{M_1})=\{(z_1,\cdots,z_N):  z_i =k^{\nu_i}y_i  \quad  \forall i, \  \mbox{where} \ y\in   \Omega(F_1) \}.
$$
Here $\nu_i=(m_i-1)/2>0$, $\beta=1+\sum_{i=1}^N\nu_i\normalcolor>0$ as in \eqref{beta}, and the parameter $k$ is related to the change of masses by $k^{\beta}=M_1$.
If $M_1>1$ this an anisotropic expansion of $\Omega(F_1)$ along all outgoing directions from the origin that we call
\begin{equation}\label{S1k}
{\mathcal S}_{k}^{(1)}(\Omega(F_1))=\Omega(F_{M_1}),
\end{equation}
 the letter ${\mathcal S}$ recalling that it is a scaling. It is very easy to see that these sets form a monotone family of sets, increasing with $k$, that cover the whole space $\ren$. If $M_1<1$ the family contracts towards the origin in a continuous way.

The same relationship applies to the free boundaries
$$
\Gamma(F_{M_1})=\{(z_1,\cdots,z_N):  z_i =k^{\nu_i}y_i \quad  \forall i, \  \mbox{where} \ y\in   \Gamma(F_1) \}.
$$

\medskip

(3) The situation of the fundamental solution $U(x,t)$ solution as $t$ moves from $t=1$ is described as follows.
The positivity set of the fundamental solution $U_1(x,t)$ at time $t>1$ is an anisotropic expansion of $\Omega(F)$ of the form
$$
\Omega(U_1,t)=\{x=( x_1,\cdots,x_N): \ (x_1\,t^{-\alpha \sigma_1},\cdots,x_N\,t^{-\alpha \sigma_N})\in \Omega(F_1)\}.
$$
It looks like the anisotropic expansion of point (2), but now the parameter $t$ is raised to a different set of anisotropic exponents, $\{a_i=\alpha \sigma_i>0\}$.
We call this expansion
\begin{equation}\label{S2t}
\mathcal{S}_{t}^{(2)}(\Omega(F_1))=\Omega(U_1,t).
\end{equation}
If $0<t<1$ the family contracts towards the origin in a continuous way.
The same relationship applies to the free boundaries.

\begin{corollary} \label{coroll.O U} There are positive constants $c_1,\cdots, c_N$ depending only on the parameters of the problem such that $\Omega(U_1,t)$ is contained in the box
$$
\Omega(U_1,t) \subset \Pi_{i=1}^N [-c_i\,t^{\alpha\sigma_i}\,,c_i t^{\alpha\sigma_i}].
$$
In the previous Corollary   the optimal constants are taken on the coordinate axes  and  the exponents in time are sharp.
\end{corollary}

\medskip

To end this part we point out that these two different expansions have produced different kinds of ``distorted balls'' but they are mutually comparable. Below we will need this fact for parameters next to 1.

\begin{proposition} \label{prop.setequiv} Let $\ve>0 $ be close to 0. There exist positive constants $c_1$ and $c_2$ such that
\begin{equation}\label{sest.comp1}
E_{1+ c_1 \ve  }(\Omega(F_1)) \subset {\mathcal S}_{1+\ve}^{(1)}(\Omega(F_1)) \subset E_{1+ c_2\ve  }(\Omega(F_1)),
\end{equation}
 where ${\mathcal S}_{k}^{(1)}$ is defined in \eqref{S1k}. There also exist positive constants $c_3$ and $c_4$ such that
\begin{equation}\label{sest.comp1}
E_{1+ c_3\ve }(\Omega(F_1)) \subset {\mathcal S}_{1+\ve}^{(2)}(\Omega(F_1)) \subset E_{1+ c_4\ve }(\Omega(F_1)),
\end{equation}
 where ${\mathcal S}_{t}^{(2)}$ is defined in \eqref{S2t}.
\end{proposition}

The proof of the first inequality relies  on the observation that the expansion $E_{1+ c_1\ve }$ increases the value of the radius $R_1({\bf e})$ by the factor ${1+ c_1\ve }$ in all directions, while the scaling ${\mathcal S}_{1+\ve}^{(1)}$ multiplies the $i$-th \normalcolor coordinate of  $R_1({\bf e})\,\bf e$ by  the factor $B_i=k^{\nu_i}$ (using the previous notations). If $k=1+\varepsilon$, then we get $B_i(\ve)\sim 1+ \nu_i \ve$ and we need to take
$$
c_1< \min\{\nu_i: 1 \le i \le N\}.
$$
Note that then $M_1=k^{\beta}\sim 1+\beta \ve$. The rest of the inequalities are similar. \qed

Note: We will use these results for factors near 1 but the argument applies also for bounded factors $k$ and $t$. We leave the details to the reader.


\section{Asymptotic behaviour}\label{sec.asymp}

Here we establish the asymptotic behaviour of finite mass solutions, another goal of this paper.

\begin{theorem}\label{thmasympto.cs}  Assume the restrictions (H1) and (H2).
Let $u(x,t)$ be the unique solution  of the Cauchy problem for equation \eqref{APM} with nonnegative initial datum  $u_{0}\in L^{1}(\R^{N})$. Let $U_{M}$ be the unique self-similar fundamental   solution with the same mass as $u_{0}$. Then,
\begin{equation}\label{L1conv}
\lim_{t\rightarrow\infty}\|u(t)-U_{M}(t)\|_{1}=\lim_{\tau\rightarrow\infty}\|v(\tau)-F_{M}\|_{1}=0.
\end{equation}
The convergence also holds in all the $L^{p}$ norms, $1< p  <  \infty$,  in the proper scale:
\begin{equation}\label{Lpconv}
\lim_{t\rightarrow\infty}t^{\frac{(p-1)\alpha}{p}}\|u(t)-U_{M}(t)\|_{p}=\lim_{\tau\rightarrow\infty}\|v(\tau)-F_{M}\|_{p}=0,
\end{equation}
where $\alpha=N/(N(\overline{m}-1)+2)$ is the constant in \eqref{alfa}.
\end{theorem}

Solution is understood in the sense of limit solutions, see Section \ref{ssec.prelim2}.

 To prove this result we proceed in two steps, first for bounded and compactly supported data, then for general integrable data\normalcolor.

\subsection{Bounded and compactly supported data}\label{sec.asymp1}

We consider  first solutions with  bounded and  compactly supported data.

\begin{proposition}\label{thmasympto.prel}  Assume the restrictions (H1) and (H2).
Let $u(x,t)$ be the unique solution  of the Cauchy problem for equation \eqref{APM} with nonnegative initial  datum  $u_{0}\in L^{1}(\R^{N})$. We also assume that $u_0$ is bounded and compactly supported in a ball of radius $R_0$. Let $U_{M}$ be the unique self-similar fundamental   solution with the same mass as $u_{0}$. Then, there are positive constants $K_1$ and  $K_2$ and a point $x_0\in\ren$ such that for every $t\ge 1$ we have
\begin{equation}\label{L1conv}
 U_{K_1}(x-x_0,t)\le u(x,t)\le U_{K_2}(x,t+1)\,.
\end{equation}
\end{proposition}

\noindent {\sl Proof. } (1) For the upper bound we observe that since $u_0(x)\le C_1$ in the ball of radius $R$ and zero outside,
there exists a large mass $K_2$ such that $u_0(x)\le U_{K_2}(x,1)$. The  comparison principle implies that the upper estimate holds for $t>0$.

(2) For the lower bound we place ourselves at $t=1$ and observe that $u(x,1)$ is a continuous and nonnegative function such that the mass is positive, $M$. It follows that $u(x,1)$ must be positive in a neighbourhood of a point $x_0\in\ren$. The estimate follows again from the  comparison principle   by comparing with a small fundamental solution centered around $x_0$. \qed

Passing to the self-similar variables
\[
v(y,\tau)=(t+1)^\alpha u(x,t),\quad \tau=\log (t+1),\quad y_i=x_i(t+1)^{-\sigma_i\alpha} \quad i=1,..,N,
\]
Proposition \ref{thmasympto.prel} can be rephrased in the following form:
\begin{corollary}\label{thmasympto.prel.v} Under the conditions of the Proposition \ref{thmasympto.prel}, if $v(y,\tau)$ is the renormalized solution we will have for all large $\tau\ge\tau_0$ that there are positive constants $K_1$ and  $K_2$ such that
\begin{equation}\label{L1conv}
 \left(1+ t^{-1}\right)^\alpha F_{K_1}\left((1+ t^{-1}\right)^{\sigma_i\alpha}(y_{i}-y_{0,i}))\le v(y,\tau)\le
F_{K_2}(y).
\end{equation}
\end{corollary}
This result shows that the functions $v(\cdot,\tau)$ approach $F$ as $\tau\to\infty$ but for some constant factors. In the sequel we eliminate this uncertainty in the constants to obtain the sharp convergence result.

\noindent {\sl Proof of Theorem \ref{thmasympto.cs} when $u_0$ is bounded and compactly supported.  }
For definiteness we assume that  $u_0$ is bounded and compactly supported in a ball of radius $R_0$. The proof uses an adaptation of the ``four-step method'', introduced by Kamin-V\'azquez \cite{KV88} and applied to the isotropic case, see also \cite[Theorem 18.1]{Vlibro}. Here we have to face substantial new issues due to the presence of degenerate anisotropy.
\medskip

\noindent {\bf $L^1$-convergence}.
We first prove the asymptotic convergence  for $p=1$:
\begin{equation}\label{L1conv.1}
\lim_{t\rightarrow\infty}\|u(t)-U_{M}(t)\|_{L^1(\ren)}=0.
\end{equation}

\medskip

\noindent (1) We introduce the family of \emph{rescaled} solutions given by
\[
u_{\lambda}(x,t)=
\lambda^{\alpha}u(\lambda^{\sigma_1\alpha}x_1,..., \lambda^{\sigma_N\alpha}x_N,\lambda t).
\]
The mass conservation,  the $L^{1}$-$L^{\infty}$ smoothing effect \eqref{Linfty-L1} and interpolation yield the uniform boundedness of the norms $\|u_{\lambda}(\cdot,t)\|_{p}$ for all $p\in[1,\infty]$ and $t>0$. Furthermore, using  \eqref{Energywholespace}  and \eqref{ai}  we have, for all $t>t_{0}>0$,\
\begin{align*}
&\int_{t_{0}}^{t}\int_{\R^{N}} \left|\frac{\partial u_{\lambda}^{m_{i}}}{\partial x_{i}}\right|^{2}dx\,d\tau=\lambda^{\alpha(2m_{i}+2\sigma_{i}-1)-1}
\int_{\lambda t_{0}}^{\lambda t}\int_{\R^{N}} \left|\frac{\partial u^{m_{i}}}{\partial x_{i}}\right|^{2}dx\,d\tau \\
&\leq \lambda^{\alpha(2m_{i}+2\sigma_{i}-1)-1}\int_{\R^{N}}\left|u(x,\lambda t_{0})\right|^{m_{i}+1}dx\le  M  \lambda^{\alpha m_i} \|u(\cdot,\lambda t_0)\|_{L^\infty}^{m_i},
\end{align*}
thus the smoothing effect \eqref{Linfty-L1} gives
\begin{equation}\label{uniformspatdervlambda}
\int_{t_{0}}^{t}\int_{\R^{N}} \left|\frac{\partial u_{\lambda}^{m_{i}}}{\partial x_{i}}\right|^{2}dx\,d\tau\le C M^{ 1+2m_i\frac{\alpha}{N}}t_0^{-\alpha m_i},
\end{equation}
an estimate that is independent of $\lambda$. Thus,   for all $i$ the derivatives $\partial_{x_{i}}(u^{m_{i}}_{\lambda})$ are equi-bounded in $L^{2}_{x,t}$ locally in time.  Moreover, the $L^{1}$-$L^{\infty}$ smoothing effect \eqref{Linfty-L1} implies that $\left\{u_{\lambda}\right\}$ is equi-bounded in $L^{\infty}$  for $t\geq\varepsilon$  with $\varepsilon>0$.

\medskip

(2) Now we use the Fréchet-Kolmogorov Theorem in order to conclude that the family $\left\{u_{\lambda}\right\}$ is relatively compact in $L^{1}_{loc}(\R^{N}\times(0,\infty))$. We have only  to control the equi-continuity of $\left\{u_{\lambda}\right\}$ in $L^1$-norm, because the equi-tightness holds since the integrals act on compact sets. \normalcolor Here is the argument. Fix $\ve>0$, there is some $\delta>0$ such that the compactness for translations holds as in the proof of Lemma \ref{lemma.precomp} for some small displacement $|z|\leq \delta$. Now for fixed times $T_{1},\,T_{2}$, a compact set $K$ of $\R^{N}$ and space-time displacements $z$, $t_{1}>0$, we can write
\begin{align*}
&\int_{T_{1}}^{T_{2}}\int_{K}|u_{\lambda}(x+z,t)-u_{\lambda}(x,t+t_{1})|dxdt\\
&\leq
\int_{T_{1}}^{T_{2}}\int_{\R^{N}}|u_{\lambda}(x+z,t+t_{1})-u_{\lambda}(x,t+t_{1})|dxdt+
\int_{T_{1}}^{T_{2}}\int_{K}|u_{\lambda}(x,t+t_{1})-u_{\lambda}(x,t)|dxdt\\
&:=A+B.
\end{align*}
Using \eqref{NewVariables} \normalcolor we can easily check that $u_{\lambda}(x,t)$ is related to the $v(y,\tau)$ by the formula
\begin{equation}\label{relationulambdav}
u_{\lambda}(x,t)=
\lambda^{\alpha}(\lambda t+1)^{-\alpha}v(\lambda^{\sigma_1\alpha}(\lambda t+1)^{-\alpha\sigma_1}x_1,..., \log(\lambda t+1)).
\end{equation}
Thus we use \ref{relationulambdav} in order to write $u_{\lambda}$ in terms of $v$. Therefore, proceeding as in the proof of Lemma \ref{lemma.precomp}, when $|z|\leq \delta/N$ we have
\begin{equation}
\begin{split}
A=&\left(\frac{\lambda}{\lambda t +1}\right)^\alpha \normalcolor
\int_{T_{1}}^{T_{2}}\int_{\R^{N}}|v\left(w_{i}+\left(\frac{\lambda}{1+\lambda t}\right)^{\alpha\sigma_i \normalcolor}z_{i},\log(1+\lambda (t+t_{1}))\right)-v(w_{i},\log(1+\lambda (t+t_{1}))|dw\,dt\\
\leq &\left(\frac{\lambda}{\lambda T_1 +1}\right)^\alpha \normalcolor\ve (T_{2}-T_{1})\leq  C_1(T_1)\ve (T_{2}-T_{1})\normalcolor.
\end{split}
\end{equation}
Moreover, observe that $v$ is uniformly continuous  in the compact set $K\times [T_{1},T_{2}]$. Observing that for $x\in K$
\[
\left(\left(\frac{\lambda}{1+\lambda t}\right)^{\alpha\sigma 1}x_{1},...,\left(\frac{\lambda}{1+\lambda t}\right)^{\alpha\sigma N}x_{N}\right)\rightarrow x
\]
as $\lambda\rightarrow \infty$ and that
\[
\log(1+\lambda t+\lambda t_{1})-\log (1+\lambda t)=\log (1+\frac{\lambda t_{1}}{1+\lambda t})\leq \delta_{1}
\]
for large $\lambda $ and small $t_{1}$, where $\delta_{1}=\delta_{1}(K,T_{1},T_{2})$ comes from the uniform continuity of $v$. Hence, for small $z$ and $t_1$ we finally have
\[
\int_{T_{1}}^{T_{2}}\int_{K}|u_{\lambda}(x+z,t)-u_{\lambda}(x,t+t_{1})|dxdt\leq  C_1(T_1)\normalcolor\ve(T_{2}-T_{1})(1+|K|).
\]
This proves the compactness of the translations of $u_{\lambda}$ on compact sets. Therefore, the Fréchet-Kolmogorov applies and we can conclude that there is some $\tilde{U}$
such that (up to subsequences)
\[
u_{\lambda}\rightarrow \widetilde{U}\text{ strongly in }L^{1}_{loc}(Q).
\]

\noindent (3) We prove that $\widetilde{U}$ is a solution to \eqref{APM}. In order to pass to the weak limit in the weak formulation for the $u_{\lambda}$'s, we use the local $L^{1}$ convergence and uniform-in-time energy estimates of the spatial derivatives  $\partial_{x_{i}}u^{m_{i}}_{\lambda}$ for all $i=1,\cdots,N$, obtained in \eqref{uniformspatdervlambda}. Therefore, using the proof of \cite[Lemma 18.3]{Vlibro} we find that $\widetilde{U}$ solves \eqref{APM} for all $t>0$. It must have a certain mass $M_1$ at each time $t>0$.

\noindent (3b) Now we prove that the mass of $\widetilde{U}$ is just $M$. We recall that we have assumed that $u_{0}$ is bounded and compactly supported in a ball $B_{R}(0) $ with mass $M$. We may use the upper bound
constructed in Proposition \ref{thmasympto.prel} as a compactly supported supersolution to make sure that no mass is lost at infinity. Indeed can argue like in the proof of \cite[Theorem 18.1]{Vlibro} and find a large self-similar solution $U_{M^{\prime}}$ such that
\[
u_{\lambda}(x,t)\leq U_{M^{\prime}}\left(x,t+\frac{1}{\lambda}\right).
\]
This estimate and the convergence $u_{\lambda}\rightarrow  \widetilde{U} $ a.e. in $\mathbb{R}^N$ allow to apply Lebesgue's dominated convergence Theorem, which yields
\[
u_{\lambda}(t)\rightarrow \widetilde{U}(t)\quad \text{in}\, L^{1}(\R^{N}).
\]
Then the mass of $\widetilde{U}$ is equal to $M$ at any positive time $t$ and  we have obtained that $\widetilde{U}$ is a fundamental solution with initial mass $M$. If such  fundamental solution is self-similar, then the uniqueness theorem would imply $\widetilde{U}(x,t)=U_M(x,t)$. If not, we can employ all the argument in \cite[Subsection 18.5]{Vlibro}
to the Lyapunov functional
$$
J[u,U_M](t)=\int_{\ren} |u(x,t)-U_M(x,t)|\,dx
$$
and conclude that $\widetilde{U}=U_{M}$.

\begin{remark}
{\rm We stress that no    B\'enilan-Crandall estimate for the time derivative $\partial_{t}u$    \cite{BC81}   is available (contrary to the isotropic case), therefore in the proof we need a novel argument to obtain relative compactness in $L^{1}_{loc}(\R^{N}\times(0,\infty))$.}
\end{remark}

\noindent {\bf $L^p$-convergence}. For every $p>1$ we can use the inequality
$$
\|u(t)-U_{M}(t)\|_{L^p}^p\le \|u(t)-U_{M}(t)\|_{L^1}\|u(t)-U_{M}(t)\|_{L^\infty}^{p-1}
$$
together with the $L^1$ convergence and the uniform estimates, $ \|U_{M}(t)\|_{L^\infty}\le C\,t^{-\alpha}$, $\|u(t)\|_{L^\infty}\le C\,(t+1)^{-\alpha}$. This argument does not apply to prove $L^\infty$ convergence.
\qed

\begin{corollary}\label{as.conv.L1.g}
The results of Theorem  \ref{thmasympto.cs} are true even if the initial data $u_0\ge 0$ are just $L^{1}$, not necessarily bounded and/or compactly supported.
\end{corollary}

We can use a density argument as in \cite[Subsection 18.2.2]{Vlibro}, where an essential tool is the monotonicity with respect to the mass devised in Proposition \ref{prop.mon.M}. The $L^{p}$ convergence follows by interpolation as in \cite[Theorem 1.2]{FVV23}.
This follows immediately from the $L^1$ contractivity of the solution semigroup, see the argument in  \cite{FVV23}.

\subsection{$L^\infty$-convergence under additional  conditions}

Actually, we have a stronger asymptotic convergence result under our initial conditions of bounded and compactly supported data.

\begin{theorem}\label{conc.as.csdata} Under the extra conditions  that $u_0$  is bounded and compactly supported\normalcolor, the convergence formula \eqref{Lpconv} holds for $p=\infty$:
\begin{equation}\label{con L infty}
\lim_{t\rightarrow\infty}t^{\alpha}\|u(t)-U_{M}(t)\|_{\infty}=0,
\end{equation}
where $\alpha$ is given by \eqref{alfa}.
\end{theorem}

We need the following quantitative positivity lemma contained in Lemma 5.1 of \cite{FVV23}, valid for fast diffusion, that can be adapted to the present slow diffusion case.
\begin{lemma}\label{lem.pos}
Let $v$ be the solution of the rescaled equation \eqref{APMs} with a nonnegative SSNI initial datum $u_0$ bounded and compactly supported with mass $M>0$. Then, there exist a constant $c>0$, a time $\tau_1>0$ and a ball of radius $r_0>0$, such that
\begin{equation*}
v(y,\tau)\ge c\quad \mbox{ for all } \ y\in B_{r_0}(0), \ \tau>\tau_1.
\end{equation*}
\end{lemma}

\noindent {\sl Proof.} The proof runs as in \cite[Lemma 5.1]{FVV23} if $v$ is controlled by an $L^1$-function. We observe that $v$ is nonnegative and SSNI. By Corollary \ref{thmasympto.prel.v} there exists $\tau_1$ and $K_2>0$ such that $v(y,\tau)\leq F_{K_2}(y)$ for all $\tau\geq \tau_1$. Then $F_{K_2}$ plays the role of barrier in the proof of Lemma 5.1 of \cite{FVV23}. It is probably more convenient to use the box $Q_{r_0}(0)=[-r_0,r_0]^N$ instead of $B_{r_0}(0)$.
\qed

\begin{remark}\label{remark pos} We observe that the lemma holds if we displace the origin and we assume that $u_0$ is SSNI around some $x_0\not\equiv0$. In order to get a convenient definition of rescaled variables $v(y,\tau)$ we have to use the shifted-space transformation $y_i=(x_i-x_{0i})(t+t_0)^{-\alpha \sigma_i}$ instead of \eqref{NewVariables}.
\end{remark}

\medskip

\noindent{\sl Proof of Theorem \ref{conc.as.csdata}.}  We can prove the result in the  equivalent formulation in terms of rescaled variables
\begin{equation}\label{conv.v. infty}
\lim_{\tau \rightarrow\infty}\|v(y, \tau)-F_{M}(y)\|_{\infty}=0.
\end{equation}

 i) First we use the $L^1$ to $L^\infty$ smoothing effect of Theorem \ref{L1LI} to conclude that  the solution $v(y,\tau)$ is uniformly bounded  from above for all $y\in\ren$, if $\tau$ is large.

 ii) Getting a uniform control from below is impossible in $\ren$ because of the property of compact support. So we want to obtain local positivity. We  argue in two steps as follows. If $u_0$ is SSNI and nontrivial, we may apply the quantitative positivity Lemma \ref{lem.pos}: \normalcolor there exist positive constants $c_1, \ve, \tau_1$ (depending on the solution) such that
\begin{equation}\label{88}
v(y,\tau)\ge  c_1 \quad \mbox{ for all } \ y\in Q_{\ve}(0), \ \tau>\tau_1.
\end{equation}
hence  $v(y,\tau)$ is uniformly bounded away from zero in a certain ball $B_\ve$ for all $y\in\ren$, $\tau$ large.

iii) In the next step we prove the analogue of \eqref{88} for solutions with a general bounded and compactly supported datum. The nonnegative solutions we consider are continuous in space and time for all positive times. Since the mass of the solution is preserved in time and the solution is continuous, then given any $t_1>0$ we may pick some $x_0\in\ren$ such that $u(x,t)\ge c_2$ for some constant $c_2>0$ in a neighborhood of $(x_0,t_1)$.  It is now convenient to change coordinates  and set the new space coordinates as $x'=x-x_0$ so $u(x',t_1)$ becomes positive at $x'=0$. The new space coordinate for $v$ will be called $y'$.

In the new coordinates we can choose a small function $w(x')$ that is SSNI around $x'=0$, it is compactly supported, and is such that $w(x' )\le u(x' ,t)$ for $ t$ close to $t_1$. For $\delta>0$ small enough let  $u_1( x' ,t)$ be the solution starting at $\bar t=t_1-\delta$ with initial value $u_1(x',t_1-\delta)=w(x')$. By comparison $u(x',t)\ge u_1(x',t)$  for all $x'$ and for $t_1-\delta<t<t_1+t_2-\delta$. We can check  that $t_2$ does not depend on $\delta$ and we recall that $t_1$ is any positive time.
Moreover, by Lemma \ref{lem.pos} and Remark \ref{remark pos} we have that the solution $u_1(x',t)$ in the rescaled variable, i.e. $v_1(y',\tau)$, is greater than or equal to $c_1$ in some box $Q_{\ve}(0)$ for $\tau>\tau_1$ for suitable $\ve,\tau_1>0$. We conclude that $v(y',\tau)\geq c_1$  for $y'\in Q_{\ve}(0)$ and for $\tau>\tau_1$.

iv) Next, we show how  we may eliminate the need for a translation. Indeed,  the present lower estimate  means that  $v(y',\tau)$ is positive in a box $Q_\ve(0)$ for all large $\tau$, so that $u(x-x_0,t)$ is continuous and  such that
\begin{equation}\label{translationestim}
u(x-x_0,t)\geq \frac{c_{1}}{(t+1)^{\alpha}}
\end{equation}
for large times and for all $x$ belonging to a suitable expanding time-depending anisotropic neighborhood $\mathcal{A}_{t}(x_{0})$ of $x_{0}$, defined through
\[
\mathcal{A}^{\ve}_{t}(x_{0})=\left\{x\in\R^{N}:\,((x_{1}-x_{01})(t+1)^{-\alpha\sigma_{1}},...,(x_{N}-x_{0N})(t+1)^{-\alpha\sigma_{N}})\in Q_{\ve}(0)\right\},
\]
that is
\[
\mathcal{A}^{\ve}_{t}(x_{0})=\left\{x\in\R^{N}:\,|x_{i}-x_{0i}|\leq (t+1)^{\alpha\sigma_{i}}\ve, \forall i=1,...,N\right\}.
\]
Now, choosing $t$ sufficiently large such that $|x_{0i}|\leq (t+1)^{\alpha\sigma_{i}}\ve$ for all $i$, we have that $0\in \mathcal{A}^{\ve}_{t}(x_{0})$ and we can put $x=0$ in \eqref{translationestim}, in order to obtain
\begin{equation}\label{translation}
u(x_0,t)\geq \frac{c_{1}}{(t+1)^{\alpha}}
\end{equation}
for $x_{0}\in \mathcal{A}^{\ve}_{t}(0)$: rephrasing this result in terms of $v$, we find inequality \eqref{88} in a box $Q_{\ve^{\prime}}(0)$, $\ve^{\prime}<\ve$.
\medskip

Points i) to iv) imply that $v(y,\tau)$  is a uniformly non-degenerate solution of the anisotropic equation in a cylinder $Q_\ve=B_\ve(0)\times (\tau_1,\infty)$ if $\ve $ is small.   By standard regularity for such equations we conclude  that $v(y,\tau)$  is H\"older continuous uniformly in space and time in the cylinder $Q_\ve$ for all large $\tau_1$.
In this situation we apply H\"older continuity and Ascoli-Arzelà to conclude that for any sequence $\tau_{n}\to\infty$ there exists some subsequence $\tau_{nk}\to\infty$ with
\[
v(y,\tau_{n_k})\rightarrow F(y)
\]
uniformly in $B_\ve(0\normalcolor)$. The uniqueness of the limit easily  implies the uniform convergence of all the family $v(\cdot,\tau)$ as $\tau\rightarrow\infty$.

(2) Now we examine the uniform convergence in the closed region away from 0 and infinity, $D_{\ve'}=\{y: |y|\ge \ve', \ F_M(y)\ge \ve'\}$ for a small $0<\ve'<\ve$ that we will choose later, where $\ve$ is as before. We argue by contradiction. Let us assume that there is  $h>0$ and a point $y_0 \in D_{\ve'}$ such that
\begin{equation}\label{Linf.contrad1}
|v(y_0, \tau_k)-F_{M}(y_0)| \geq h>0
\end{equation}
along a sequence $\tau_k\to\infty$. To continue we recall  the $L^1$ convergence result just proved, so that for any $\delta>0$  we may take $\tau_1$ large enough so that $\|v(y, \tau)-F_{M}(y)\|_1<\delta$ for every $\tau>\tau_1$. Recall that $F_{M}$ is smooth  on  $D_\ve'$ and we have
$F_{M}(y_0)\ge \ve'$.

\smallskip

(2-i) A first possibility is that  the separation in \eqref{Linf.contrad1} happens downwards:
$$
v(y_0,\tau_k)\le  F_{M}(y_0)-h
$$
for all $\tau_k\to \infty$. By the continuity property of $F_{M}(y)$ we may find a neighborhood $B \subset D_\ve'$ of $y_0$ where
$$
F_{M}( y  )> F_{M}(y_0) -(\ve'/2) \geq\ve'/2 \quad  \text{ for }y\in B.
$$

Now we want to use the partial monotonicity property  of $v(\cdot,\tau)$. Recalling that our initial datum $u_0$ is supported in a box $Q(a)$ of sizes $a=(a_i)$, then Corollary \ref{Cor 3b} applies for all $x_i>a_i$. This means that in term of $v$ the monotonicity holds for $y_i> a_i(t+1)^{-\alpha \sigma_i}$, i.e. outside a contracted box, whose size tends to zero as $t\to\infty$. So the region of (possible) no monotonicity shrinks enough with time and it is contained in a ball centered in the origin of radius $\varepsilon'$ for times big enough. In conclusion by the partial monotonicity property of Corollary \ref{Cor 3b} applied to $v(\cdot,\tau)$ outside the ball of radius  $\varepsilon'$, we find a cone $K$ of positive directions where
$$
v(y,\tau_k)\le  v (y_0,\tau_k)  \qquad \mbox{for all } y\in y_0+K.
$$
In the intersection $ B\cap (y_0+K)$ we get, up to choosing $\ve'<h$,
$$
F_M(y)-v(y,\tau_k)  \geq \left[F_M(y_0)-\ve'/2\right]-v (y_0,\tau_k) \geq h-\ve'/2\ >\ve'/2.
$$
This uniform bound is valid in a fixed set of nonzero measure so that it contradicts the $L^1$ convergence in $\ren$. We conclude that
$v(y,\tau)$ cannot separate from $F_M(y)$ from below as in \eqref{Linf.contrad1}.

\smallskip

(2-ii) In the second step we show that, in the same setting, $v(y,\tau)$ cannot separate from $F_M(y)$ upwards either.   We again argue by contradiction, now we examine any point  $y_0 \in D_{\ve'}$ where
$$
v(y_0, \tau_k)-F_{M}(y_0) \geq h>0
$$
along a sequence $\tau_k\to\infty$. The estimate of difference in a small uniform neighborhood uses  the partial monotonicity property  of $v(\cdot,\tau)$ (see  Corollary \ref{Cor 3b}). Then  we find a cone $K$ of negative directions where
$$
v(y,\tau_k)\ge  v (y_0,\tau_k)  \qquad \mbox{for all } y\in y_0+K.
$$
We conclude as before.

\smallskip

(3) The same argument for small separation from above applies in the set $E_{\ve'} =\{y:\,|y|\geq\ve',\, \ F_M(y)\le \ve'\}$  for a suitably small $0<\ve'<\ve$.
The lack of separation from below is trivial since $v\ge 0$.
Now the proof is complete.   \qed


\section{Expansion of the support for large times}\label{sec.exr.supp}

 Here we examine how the spatial support of a nontrivial solution $u\ge0$ with compactly supported and bounded initial data evolves in time.
In doing that we continue the study of Section \ref{ssec.support} for the fundamental solution. Here we consider a solution $u(x,t)$ under the conditions of Theorem \ref{thmasympto.cs} and denote by $v(y,\tau)$
the  corresponding renormalized  solution.

Let us recall some notations for the sets that will deal with, the type that was actually adopted in Section \ref{ssec.support} when dealing with self-similar profiles. For a continuous function $f=f(x)$ defined in a set $A\subset \ren$  we denote the positivity set by
$$
\Omega(f)=\{x\in A:  f(x)>0\}.
$$
When $f=F_M$, the fundamental profile, as in Section \ref{ssec.support},  we write $\Omega(F_M)$ for its positivity set.  For the time-dependent functions $U_M(x,t)$ and $u(x,t)$ we write
$$
\Omega(U_M,t)=\{x\in \ren:  U_M(x,t)>0\},  \quad  \Omega(u,t)=\{x\in \ren:  u(x,t)>0\},
$$
where $t>0$ acts as a parameter. After renormalization to $v(y,\tau)$ we consider the positivity set $\Omega(v,\tau)=\{y\in \ren:  v(y,\tau)>0\}$ as a function of $y$  for fixed $\tau>0$.

We have already introduced in Section \ref{ssec.support}  a first scaling transformation  \eqref{S1k} that allows to pass positivity sets from any mass $M$ to mass one, i.e., from  $\Omega(F_M)$ to $\Omega(F_1)$:
$$
\mathcal{S}_k^{(1)}(\Omega(F_1))=\Omega(F_M).
$$
These sets  form a monotone family of sets, continuously increasing with $k$,  that cover the whole space $\ren$.
 The $\Omega(F_M)$ expand along all outgoing directions with an anisotropic rate.
The second scaling transformation \eqref{S2t} relates
$\Omega(U_M,t)$ and $\Omega(F_M)$
$$
\mathcal{S}_t^{(2)}(\Omega( F_M))= \Omega(U_M,t).
$$
Now the $\{\Omega(U_M,t) : t > 0\}$ form an expanding  family of sets that cover
the whole space $\ren$. The $\Omega(U_M,t)$ also expand along all outgoing directions, with a different anisotropic rate.
Finally, the relation \eqref{NewVariables} between $u(x,t)$ and $v(y,\tau)$ implies that
$$
x=( x_1,\cdots,x_N)\in \Omega(u,t) \ \mbox{iff} \ y=( y_1,\cdots, y_N)\in  \Omega(v,\tau) \ \mbox{with} \ x_i=(t+t_0)^{\alpha\sigma_i}y_i,
$$
which is only a small variation of the latter scaling, without importance for large $t$.

\smallskip

In the sequel we will examine the behaviour of the positivity set and the support of the solutions $u$ and $v$ for large times. They will increasingly resemble  the self-similar sets, in the equivalent senses that $\Omega(u,t)$ increasingly resembles $\Omega(U_M,t)$ as $t\to\infty$, and that $\Omega(v,\tau)$ resembles $\Omega(F_M)$ as $\tau\to\infty$.

First, we use Corollary \ref{thmasympto.prel.v} to prove a weaker convergence,  up to constant factors.

\begin{proposition}\label{exp.cs.estim}  Under the conditions of  Proposition \ref{thmasympto.prel} , there are constants $M_1<M<M_2$ such that estimates
\begin{equation}\label{supp.estim11}
\Omega(F_{M_1}) \subset \Omega(v,\tau) \subset \Omega(F_{M_2}),
\end{equation}
if the time $\tau$ must be large enough. By displacement of the origin of time we may assume that this happens for all $\tau>0$ and also that
\begin{equation}\label{supp.est-v}
F_{M_1}(y) \le  v(y,\tau) \le F_{M_2}(y).
\end{equation}
\end{proposition}

We stress that the upper bounds are consequence of Corollary \ref{thmasympto.prel.v}, but the lower bounds follow from the $L^\infty$ version of the asymptotic result
Theorem \ref{conc.as.csdata}.

\begin{figure}[t!]
 \centering
 \includegraphics[width=12cm]{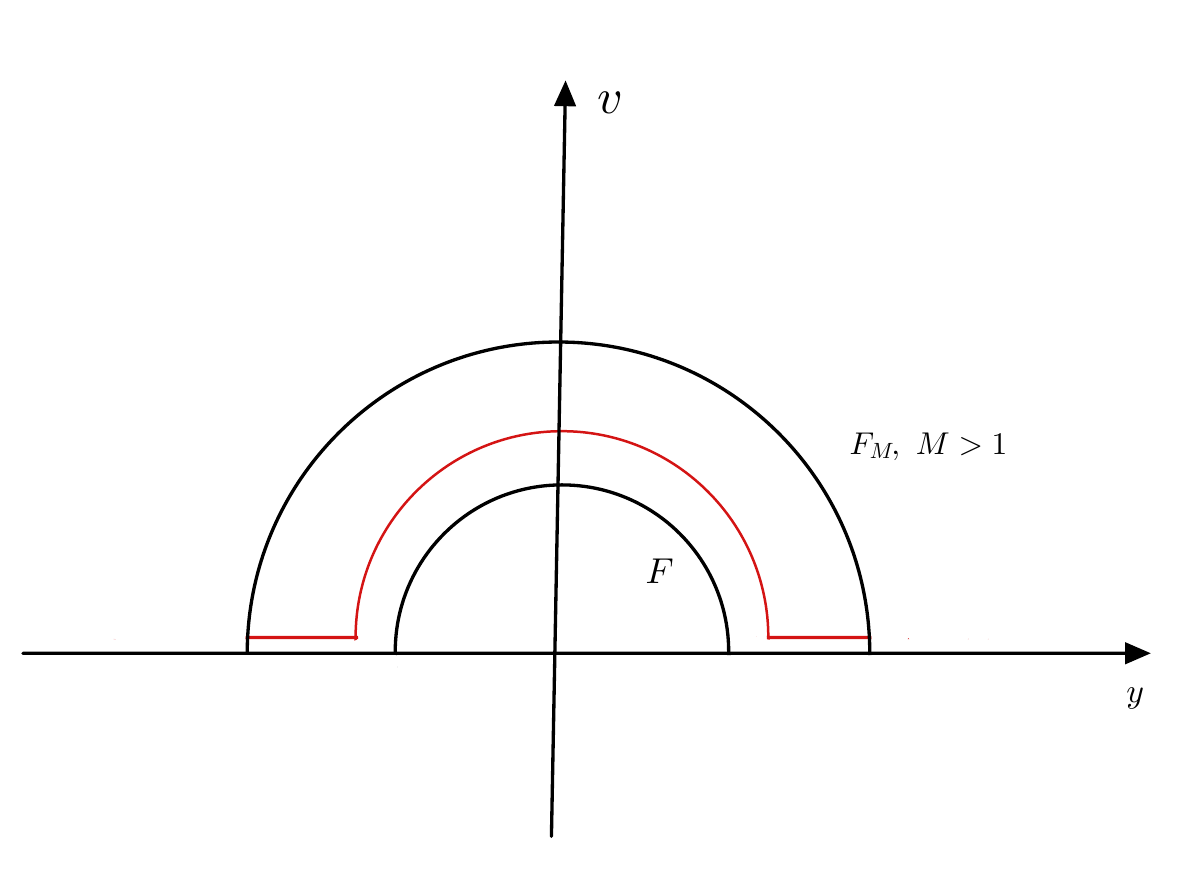}
 \caption{Schematic situation for large times:  The red line is the upper bound for $v(y,\tau)$ for large times. Only one space direction is represented for simplicity. }
\end{figure}

Next, we will prove a much sharper approximation result that needs a finer study in terms of the Hausdorff set distance. We recall that for two sets $A,B\subset\ren$ we define this distance $d_H(A,B)$ as follows. For any $a\in A$, $b\in B$ we the first define the distance from point to set
$$
d(a,B)= \min \{d(a,b):  b\in B \},  \quad
d(b,A)= \min \{d(a,b): a\in A\},
$$
and  then we define the two directional distances
$$
d_1(A,B)=\sup\{d(a,B): a\in A),
$$
which is the sup of the distances from points of $A$ to the set $B$, and
$$
\quad d_2(A,B)=\sup\{d(b,A): b\in B).
$$
where the sup interchanges $A$ and $B$. Then the symmetric distance is defined as
$$
d_H(A,B)=\max\, \{ d_1(A,B), d_2(A,B) \}.
$$
We will work with bounded sets.

\noindent {\bf Remarks.}  (1) It is usual to consider closed sets in these definitions, but in our case we do not have to worry (take closures to pass from positivity sets to supports).

\noindent (2) With the above definitions, it is immediate to see that the family $\{\Omega(F_M): M>0 \}$ is continuous with respect to $M$ and its size increases in a power way with  $M$ in all directions. Moreover, the family  $\{\Omega(U_M,t) : t > 0\}$  is continuous with respect to $t$ and its size increases in a power way with $t$ in all directions.
Moreover, it happens that for $k\sim 1$ and $M\sim 1$
both $\Omega(U_M,t) $ and $\Omega(F_M)$ are comparable if $k-1$ and $t-1$ are small and proportional, see Proposition \ref{prop.setequiv}.

\begin{theorem}\label{exp.cs}  Under the conditions of Proposition \ref{thmasympto.prel}, we have the  set convergence
\begin{equation}\label{supp.estim.sharp}
\lim_{\tau\rightarrow\infty} d_H(\Omega(v,\tau), \Omega(F_M))=0.
\end{equation}
 The same formula works for the supports, $d_H(S(v,\tau), S(F_M))\to 0$.
\end{theorem}{}


\subsection{Proof of Theorem \ref{exp.cs} }

By using the  mass rescaling of solutions studied in Subsection \ref{sec scaling} we may assume without loss of generality that $M=1$ and write $F_M= F_1$. We  will  call $\Gamma(F_1)$ the boundary of \ $\Omega( F_1)$, or equivalently, the boundary of the support of $F_1$ (since $S(F_1)$ is the closure of $\Omega(F_1)$).

(1) We want to establish first the approximation of  the family of sets \
$\Omega(v,\tau)$ \ to \ $\Omega(F_1)$ from inside as $\tau$ goes to infinity. We mean that the points of $\Omega(v,\tau)$  cover all the space  inside of $\Omega(F_1)$ but for points at a small distance of $\Gamma(F_1)$, and that distance goes to zero with $\tau$. In other words, we want to prove that
\begin{equation}\label{Th 7.2 formula 1}
d_1(\Omega(F_1), \Omega(v,\tau)) \to 0 \qquad \mbox{as } \ \tau\to 0.
\end{equation}
Let us introduce the interior subsets $G_\ve =\{y: F_1(y)\ge \ve\} $ for $\ve>0$. We easily conclude from the uniform convergence result in Theorem \ref{conc.as.csdata} that for all large $\tau$ we have
$$
\Omega(v,\tau)\supset G_\ve,
$$
in other words $d_1(G_\ve, \Omega(v,\tau))=0$.    On the other hand,
$$
d_1(\Omega(F_1), G_\ve) = d_1(\Gamma(F), G_\ve)\to 0 \qquad \mbox{as } \ \ve\to 0,
$$
\noindent which follows from elementary topology for continuous functions. By the triangle inequality
$$
d_1(\Omega(F_1), \Omega(v,\tau)) \to 0 \qquad \mbox{as } \ \tau\to 0,
$$
This is the desired statement, there are no points of $\Omega(F_1)$ that stay far from $\Omega(v,\tau)$ as $\tau\rightarrow\infty$.

\medskip

(2) The approximation of  $\Omega(v,\tau) $  to $\Omega(F_1)$  from outside,
\begin{equation}\label{supp.d21}
 d_2(\Omega(F_1), \Omega(v,\tau))  = d_1(\Omega(v,\tau), \Omega(F_1)) \to 0 \qquad \mbox{as } \ \tau\to 0,
\end{equation}
is more delicate and takes several steps. Let us present the difficulty that might arise in this kind of limit problems: a family of positivity sets of a converging family of functions may fail to approximate the limit when the convergence of supports is examined, and this is because of the presence of  ``thin positivity tails'' of evanescent intensity that must be however counted  in the sense of positive sets even if their mass is small.

The idea we use to avoid such failure of  set convergence (i.e., convergence in the sense of Hausdorff distance of sets) is to combine the already established uniform convergence of $v(y, \tau)$ in the positivity set of $F_1$ with a new comparison in the exterior of $\Omega(F_1)$ with an astute upper barrier.
In that direction we recall that by our starting assumption in Proposition \ref{exp.cs.estim} is that there exists a mass $M_2>1$ such that $ v(y,\tau)\le F_{M_2}(y)$ and then $\Omega(v,\tau)$ is contained in  $\Omega(F_{M_2})$  for $\tau$ big enough. This means that for all large $\tau$ the solution vanishes outside a large set, more precisely
\begin{equation}\label{supportM2}
v(y,\tau)=0 \quad \mbox{ for } \ y\not\in \Omega(F_{M_2}).
\end{equation}
 We only have to reduce the role of the mass $M_2$ in these statements to a smaller size of mass $M_3$ near $M=1$ for the desired result \eqref{supp.d21} to follow. The reduction part is delicate and is proved next.

\begin{proposition}\label{reduction}  In the present situation, for every $M_3>1$ there exists a time $\tau_1$ such that we have the estimate
\begin{equation}\label{supp.estim111}
v(y,\tau)\le F_{M_3}(y) \quad  \text{ for all }  y\in \ren  \ \text{ and }\tau\ge \tau_1.
\end{equation}
It follows that $\Omega(v,\tau)\subset \Omega(F_{M_3})$ for all $\tau$  large enough.
\end{proposition}
\medskip

Recall that we are working with solutions of mass 1 and then the desired improved bound $M_3>1$ is close to 1. Moreover $M_3$ has as upper bound $M_2$, the mass that appears in from Proposition \ref{reduction}.

Step 2 of the Theorem \ref{exp.cs} follows immediately by Proposition \ref{reduction} recalling that
 $$
 d_1(\Omega(F_{M_3}),\Omega(F_{1}))\to 0 \quad \mbox{as } M_3\to 1.
 $$
 Then the proof of Theorem \ref{exp.cs} is complete if Proposition \ref{reduction} is proved.

 \noindent{\sl Proof of the Proposition \ref{reduction}}.
The proof is divided in some steps.

(1) By the already mentioned expansion property of the self-similar solution with mass $1$, we may find a time shift $s_0$ such that
the support of $U_1(x, s_0)$  strictly includes $\Omega(F_{M_2})$ (which in turn is a fixed expansion of $\Omega(F_1)$ by scaling $\mathcal{S}_k^{(1)}$  defined in \eqref{S1k}).
Then we may find $\ve_0 >0$ such that in $\Omega(F_{M_2})$ we have uniform positivity
\begin{equation}\label{U varepsilon}
U_1(x,s_0)= s_0^{-\alpha}F_1( x_1\,s_0^{-\alpha \sigma_1},\cdots, x_N\,s_0^{-\alpha \sigma_N})\ge \ve_0,
\end{equation}
\noindent where $\ve_0>0$ depends on $s_0$.
We will use this function to define the useful barrier $\tilde u(x,t)$ with origin of time at $t=0$, using shifting in time, i.e.,
$\tilde u(x,t)=U_1(x,t+s_0).$ The corresponding $v$-solution is
$$
 \tilde  v(y,\tau)=(t+1)^{ \alpha}\tilde u(x,t), \quad x_i=y_i(t+1)^{\alpha \sigma_i}, \quad \tau=\log(t+1).
$$
Working out the details we get
$$
\tilde  v(y,\tau)=((t+s_0)/(t+1))^{-\alpha}F_1(
y_1\,(t+1)^{\alpha \sigma_1}(t+s_0)^{-\alpha \sigma_1},
\cdots, y_N\,(t+1)^{\alpha \sigma_N}(t+s_0)^{-\alpha \sigma_N} ).
$$
Notice that the expansion factor is now $(t+s_0)/(t+1)\to 1$ as $t\to\infty$,  then  the positivity set of $\tilde  v(y,\tau)$ shrinks with time towards $\Omega(F_1)$ as $\tau\to\infty$.

\begin{figure}[t!]
 \centering
 \includegraphics[width=12cm]{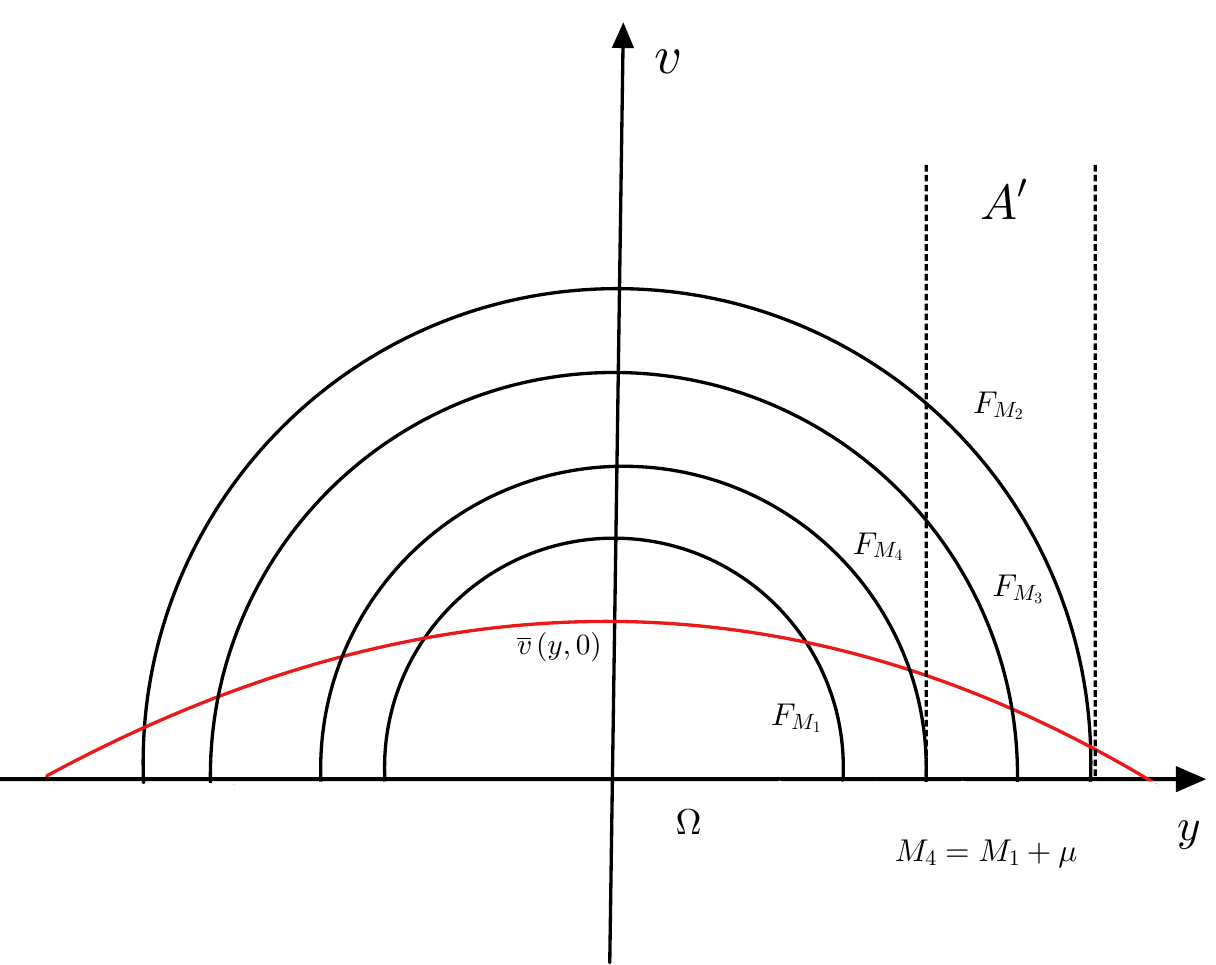}
 \caption{Schematic situation for the comparison proof} \label{anullus}
\end{figure}

Indeed, the support of $\tilde  v(y,\tau)$ is an expansion of the support of $ F_1$ (with scaling of type $\mathcal{S}_k^{(2)}$  defined in \eqref{S2t}) and the argument is just observing that the distance between any point $y$ in the support and the expanded point
$$
\tilde y= (y_1\,((t+s_0)/(t+1))^{\alpha \sigma_1}
\cdots, y_N\,((t+s_0)/(t+1))^{\alpha \sigma_N})
$$
goes to zero uniformly with $t\to\infty$.

(2) In the next step we compare $\tilde v (y,\tau)$ with the original solution $v$ outside of $\Omega(F_{1})$. This comparison is delicate and needs several modifications. First, we will compare  after performing on $v$ a time displacement of the form
$$
v_k(y,\tau)=v(y,\tau+k)
$$
for some $k>0$ possibly very large (see the argument below to find the role of $k$).

A very important technical detail concerns the domain where comparison can be successfully performed. We want to compare in a region where $v_k$ is small for large $k$. This region cannot contain $\Omega(F_{1}) $ since  $v_k$ converges uniformly to $F_1$ which is positive there (by Theorem \ref{conc.as.csdata}). Moreover, by the uniform convergence in this region we can assure that for $k$ large and for all $y\in \Omega(F_{1})$
\[
v_{k}(y,\tau)<F_{1}(y)+\ve
\]
where we can choose $0<\ve\le \min_{\Omega(F_{1})}[F_{M_3}(y)-F_{1}(y)]$, thus
\begin{equation}
v_{k}(y,\tau)<F_{M_3}(y) \text{ for all }y\in \Omega(F_{1}).
\label{compF_{3}}
\end{equation}

The first idea is to make a comparison in a region where $v_{k}$ is small is to study  the anisotropic annular region $A=\Omega(F_{M_2})\setminus \Omega(F_{1})$, see Figure \ref{anullus}. This is essentially correct, but later we  find a difficulty  that forces us to use a slightly smaller annular domain, $A'=\Omega(F_{M_2})\setminus \Omega(F_{M_4})$  for certain times $0<\tau<\tau_1$. Here   $\tau_1$  will be large as needed and
$M_4>1$ must be much closer to 1 than $M_3$ (we will adjust all the quantitative aspects during the proof).

We are ready to check the conditions for parabolic comparison.

$\bullet$  We recall that  in  $A'$ we  have uniform positivity for  $\tilde   v$ at $\tau=0$  by
\eqref{U varepsilon}:
$$
\tilde   v (y,0)=s_0^{-\alpha}F_1( y_1\,s_0^{-\alpha \sigma_1},\cdots, y_N\,s_0^{-\alpha \sigma_N})\ge \ve_{\magenta0}.
$$

 Because  of the uniform convergence (from the asymptotic Theorem \ref{conc.as.csdata}) we can take $k$ large enough so  that
$$
v_k(y,0)=v(y,k)\le \delta  \quad \mbox{in } A'
$$
for some $\delta < \ve_{\magenta0}\le \tilde  v (y,0)$. Recall that in this region $F_{1}(y)=0$. This $\delta$ will be subject to another condition below.

(ii) On the outer lateral  boundary, that is on the boundary of $\Omega(F_{M_2})$, we have
$v_k\le  \tilde v $  for $k$ big enough  because $v_k$ is zero by virtue of the upper estimate
in Proposition  \ref{exp.cs.estim}, see \eqref{supportM2}.

(iii) A more delicate argument is needed at the inner lateral boundary, the boundary of $\Omega(F_{M_4})$, where $v_k$ is maybe not zero. Let us examine both functions. We know that $v_k$ is uniformly small outside of $\Omega(F_1)$, $0\le v_k\le \delta$ when $k$ is large enough, because of the asymptotic  result of Theorem \ref{thmasympto.cs},  and we may take $\delta$ much smaller than $\ve_{0}$. On the other hand, we have said that the positivity set of  $\tilde  v (y,\tau)$  shrinks to become uniformly $\Omega(F_{1})$ as $\tau\to\infty$, so the comparison must fail when  $\tau$ is too large.

(iii') This difficulty is repaired by fixing an upper bound for $\tau$ in a quantitative way as follows. Let us fix $M'=1+\mu$ (very close to 1). We wait  until  the border of the (shrinking family) $\Omega(\tilde v,\tau)$ touches (from outside) the border of $\Omega(F_{M'})$, say at the time $\tau_1$.
Recall that both anisotropic sets have different shapes but they are comparable by Proposition \ref{prop.setequiv},  up to constant factors that depend on the $m_i$.
Indeed, since we take $\mu$ very small, this means that  $\Omega(F_{M'})$ is a very small expansion of $\Omega(F_1)$ with expansion  factors comparable to $1+c_i\mu$ in every direction, which means that $\tau_1=\tau_1(\mu)$ must be very large, depending on $\mu$. Then, $\Omega(\tilde v,\tau_1)$ is also a very small expansion of $\Omega(F_1)$ with factors comparable to $1+d_i\mu$ in every direction.

For later use, we also need $\mu$  to be small enough so that $\Omega(\tilde  v,\tau_1)\subset \Omega(F_{M_3})$ (recall that $\tilde  v(y,\tau)$  tends towards $\Omega(F_1)$ as $\tau\to\infty$ in a uniform  way).

(iii'') We now fix the inner boundary as $\Gamma(F_{M_4})$ with $M_4<M'$ so close to 1 that (see Figure \ref{SetInclusions})
\begin{equation}\label{99}
\Omega(F_{M_4}) \subset \Omega(\tilde v,\tau_1+1).
\end{equation}
\begin{figure}[t!]
 \centering
 \label{figureanullus}
 \includegraphics[width=10cm]{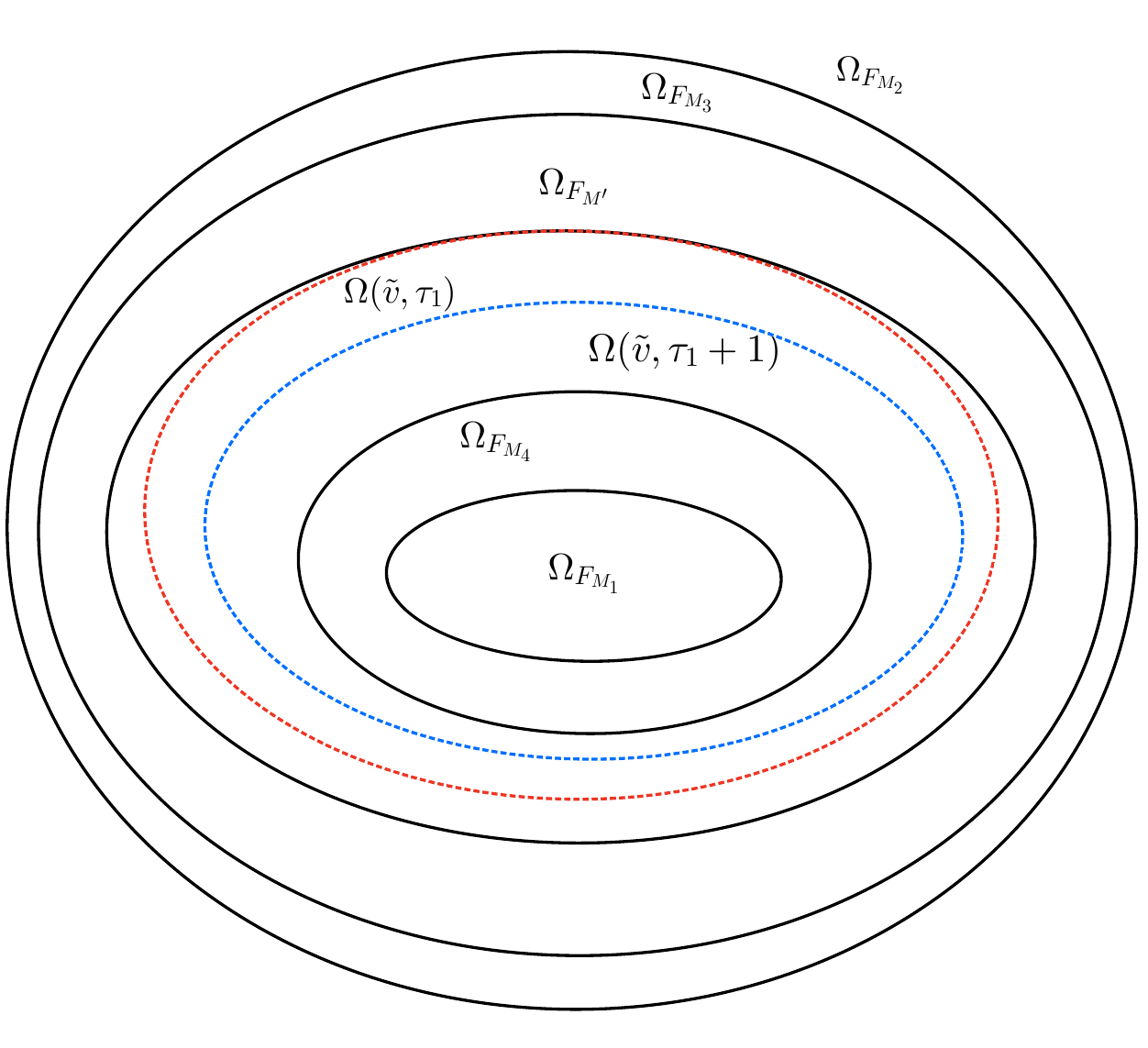}
 \caption{Scheme of the dynamics of $\tilde{v}$ at time $\tau_{1}$}\label{SetInclusions}
\end{figure}
Since $\Omega(\tilde v,\tau_1+1)$ is strictly contained in $\Omega(\tilde v,\tau_1)$ we see that  $\tilde  v (y,\tau_1)$ is strictly positive in $\Gamma(F_{M_4})$, and the same bound is true for all times $0<\tau<\tau_1$ by inspection of the formula (which deals with monotone values at shrinking points). Indeed, the ratio
\[
g(\tau)=\frac{e^{\tau}}{e^{\tau}+s_{0}-1}
\]
entering in the definition of $\tilde{v}$ is increasing for $s_{0}>1$, therefore for $\tau<\tau_{1}$, by the SSNI property of $F_{1}$ we have
\[
\tilde v (y,\tau)\geq \frac{1}{(t_{1}+1)^{\alpha}}\tilde{v}(y,\tau_{1}):=\delta,
\]
where $\tau_{1}=\log(1+t_{1})$.
This implies that $\tilde  v (y,\tau)$ will be uniformly larger than $\delta>0$  on the inner boundary for all $0<\tau\le\tau_1$, and  $\delta$ is small enough depending on the $\mu$ and $M_4$ chosen above.

Recalling that by points (i) and (iii) we were free to chose the $\delta$ appearing in the estimate for $v_k$, we may apply the
parabolic comparison principle to conclude that  $v_k(y,\tau)\le  \tilde  v (y,\tau)\blue$ in $A'\times (0,\tau_1)$  for large $k$.
It follows that there exists $\tau_1$  such that
 $$
 \Omega(v,\tau_1+k)\subset \Omega(\tilde  v,\tau_1)\quad  \hbox{for large } k.
 $$
 Indeed, if $y\in \Omega(v,\tau_1+k) $ and $y\in \Omega(F_{M_4})$ then  by \eqref{99}  we have $y\in  \Omega(\tilde v,\tau_1+1)\subset  \Omega(\tilde v,\tau_1) $. If $y\not\in \Omega(F_{M_4})$,  we must have $y\in \Omega(F_{M_2}) $ (because of \eqref{supportM2}), hence $y\in A'$ and by the comparison $0<v(y,\tau_{1}+k)\leq \tilde v(y,\tau_1) $.

(iv) To end the comparison of supports we observe that for $\mu\to 0$ we have $\tau_1\to\infty$ so using $\mu$ small we may find
\begin{equation}
 \Omega(\tilde  v,\tau_1)\subset \Omega(F_{M_3}), \qquad  \tilde v(y,\tau_1)\le F_{M_3}(y) \quad \mbox{for all } y.\label{comptauM3}
\end{equation}
This follows from examining the scalings with small parameters involved in $\tilde v$ and $U_{M_3}$ and using
Proposition \ref{prop.setequiv} that proves that both scalings in time and mass  are mutually comparable.  Putting things together, we conclude that
 \[
 \Omega(v,\tau_1+k)\subset\Omega(\tilde  v,\tau_1)\subset \Omega(F_{M_3})\quad  \hbox{ for } k \hbox{ large}.
 \]

(vi) To finish the argument we want to make sure that
$$
 v(y,\tau_1+k)\le F_{M_3}(y) \quad \mbox{for all } y
 $$
and the same large $k$. Take $y\in \Omega(v,\tau_1+k)$. If $y\in\Omega(F_{1})$ then by \eqref{compF_{3}} the previous inequality holds for $k$ large. If we have $y\in A'$ the comparison between $v_{k}$ and $\tilde{v}$ holds, so it works (see \eqref{comptauM3}). Finally, if $y\in \Omega(F_{M_4})\setminus \Omega(F_{1})$, we have for large $k$ that $v_{k}(y,\tau_1)\le \delta\le F_{M_3}(y)$.

 (vii) Once the inequality between two solutions (i.e., $ v(y,\tau_1+k)\le F_{M_3}(y)$ holds for all $ y$)
 is shown to hold in the whole space at one time $\tau_1+k$, for conveniently large $k$ and $\tau_1$, it will hold forever afterwards,
 for $ \tau\ge \tau_1$.

 For the readers's convenience we record the choosing order for the different constants in the proof: given the solution $v$ and the masses $M_2>M_3>1$, we have chosen first $s_0$, then $\ve_{0}$, then $\mu$ quite small, then $\tau_1$, then $M_4$ and $A'$, then $\delta$ and $k$.
 \qed

\subsection{Convergence of the free boundaries}\label{sec.asympII}

We now formulate the previous asymptotic result in terms of the free boundaries. As in Proposition \ref{thmasympto.prel} and  Theorem \ref{exp.cs} we consider a solution  $u(x,t)$  of the Cauchy problem for equation \eqref{APM} with nonnegative initial  datum  $u_{0}\in L^{1}(\R^{N})$. We also assume that $u_0$ is bounded and compactly supported in a ball of radius $R_0$. Both restrictions (H1) and (H2) on the exponents of  \eqref{APM} are assumed. We let $v(y,\tau)$ be the renormalized solution according to \eqref{NewVariables}.

\begin{theorem}\label{exp.fb}  Under the above conditions, for every $\ve>0 $ there is a time $\tau(\ve)$  large enough such that for $\tau\ge \tau(\ve)$
\begin{equation}\label{supp.fb1}
\Gamma(v,\tau)\subset \Omega(F_{M+\ve})\setminus \Omega(F_{M-\ve})\,,
\end{equation}
where $M>0$ is the mass of $u_0$. It follows that
\begin{equation}\label{supp.estim.sharp}
\lim_{\tau\rightarrow\infty} d_H(\Gamma(v, \tau), \Gamma(F_M))=0.
\end{equation}
\end{theorem}

\noindent{\sl Proof.} Without loss of generality we may assume that $M=1$ as we have done in previous sections. The inclusion $\Gamma( v,\tau)\subset \Omega(F_{1+\ve})$ is a direct consequence of
Proposition \ref{reduction}. For the other inclusion we start from \eqref{Th 7.2 formula 1}
$$
d_1(\Omega(F_1), \Omega(v,\tau)) \to 0 \qquad \mbox{as } \ \tau\to 0,
$$
that is easily translated into
$$
d_1(\Gamma(F_1), \Omega(v,\tau)) \to 0 \qquad \mbox{as } \ \tau\to 0,
$$
and then into
$$
d_1(\Gamma(F_1), \Gamma(v,\tau)) \to 0 \qquad \mbox{as } \ \tau\to 0.
$$
By the expansion properties of the family $F_{\blue 1 \magenta}$ we easily see that there is a constant $c$ such that
$$
d_H(\Gamma(F_1), \Gamma(F_{1-\ve})= d_H(\Gamma(F_1), \Omega(F_{1-\ve})\ge c\ve.
$$
Therefore, the triangle inequality implies that for large $\tau\ge \tau(\ve)$ we find that
$\Gamma(v,\tau)$ is disjoint with $\Omega(F_{1-\ve})$, and also with $\Gamma(F_{1-\ve})$. At this point the last assertions of the theorem follows easily.  \qed


\section{Comments and open problems}\label{sec.comments}

$ \bullet  $ {\bf The isotropic case.}  The main results of our paper are well-known in the isotropic case $u_t=\Delta u^m$. The rigorous mathematical study started around 1950  with the work of Barenblatt et al., \cite{Bar52}. and it took around 50 years to be reasonably complete, it is covered in the book \cite{Vlibro}. It includes the theory of existence and uniqueness of the different initial and boundary value problems, the basic estimates and inequalities and the theory of regularity. For comparison with the main results of this paper, we mention that the existence of a self-similar solution taking a Dirac delta as an initial function is guaranteed by the explicit Barenblatt solution
 $$U_M(x,t)=t^{-\alpha}F_M(t^{-\alpha/N}x),
 $$
 where $F_M$ is given by formula \eqref{Fm} with parameter $\alpha=N/(N(m-1)+2)>0$.
 This is what we call the SSF solution. Its uniqueness is also guaranteed in the class of weak solutions. This has to be compared with Theorem \ref{fundamental solution}
of our paper, which needs much more elaborate arguments. The first main difference is this one: in isotropic case the solutions are radially symmetric with respect to  the space variables, so the existence calculations become one-dimensional.

Concerning the asymptotic behaviour that we describe in Theorems \ref{thmasympto.cs}  and \ref{conc.as.csdata}, the isotropic version is a classical topic in the nonlinear diffusion theory after the works of Kamin \cite{Kam73, Kam76}. The isotropic result is studied in great detail in
in Chapters 18 and 19 of  \cite{Vlibro}.

Finally, our crowning result on the large-time behaviour of supports, Theorem \ref{exp.cs} is deduced in the isotropic case from an argument that is simplified by the fact that the support of the profile is just a ball of radius $R$ (depending on the total mass $M$). Then the supports of $v(\cdot,\tau)$ are proved to be approximate balls, see Theorem 18.8 of \cite{Vlibro}.
 In geometrical terms, we interpret such result in the isotropic PME case as saying that the support of a general solution with finite mass and compact initial support evolves in time to fill a ball of radius $R(t)$ given by the closest Barenblatt solution. This known result extends in the present anisotropic case to the result of Theorems \ref{exp.cs} that the support evolves in time to fill a distorted ball which is given by the closest SSF solution, and the approximation is exact in relative error.  So the distorted balls are stable geometries  for the anisotropic model. Theorem  \ref{exp.fb} express this fact in terms of convergence of free boundaries.

We continue by remarking that in the isotropic case the relative error of the support of $v(t,\tau)$ w.r.t. $B_R(0)$ can be estimated in a very precise way as an asymptotic expansion that includes rates for the vanishing of the error. For recent results see \cite{KKV16, Seis2014}. This is open territory in our case. An so is the question of regularity of the free boundaries.

\medskip

$ \bullet  $ {\bf The fast/slow anisotropic case.} The same anisotropic Porous Medium Equation \eqref{APM} with fast directions $m_i\le 1$ mixed with slow directions $m_i>1$ is quite interesting but deserves it own paper, to be developed.

$ \bullet  $ {\bf The $p$-Laplacian model.} A similar analysis should apply to the slow diffusion regime for the anisotropic $p$-Laplacian model
\begin{equation}\label{APLs}
\partial_t u=\sum_{i=1}^N  \partial_i(|\partial_i u|^{p_i-2}\partial_i u),
\end{equation}
with $p_i>2$ for all $i$ (slow diffusion). The  computations are similar but the justification must be done. The case where some $p_i$ equals 2 is interesting.

 Vespri and collaborators have been working  this model with $p_i>2$ (slow diffusion), and some structural conditions, see  \cite{CV21, CMV23} and references. The type of questions they address is different and so are the tools. Thus, they discuss regularity questions  like the parabolic version of the Harnack inequality and the intrinsic geometry. They also study the existence and properties of the Barenblatt solutions but not the uniqueness, free boundaries or asymptotic behaviour. The fast diffusion version of that model was studied by us in \cite{FVV21}, the different type of propagation speed implies quite different qualitative behaviour.


\section*{Acknowledgments}
F. Feo and B. Volzone are partially supported by GNAMPA of the Italian INdAM (National Institute of High Mathematics) and by the "Geometric-Analytic Methods for PDEs and Applications" project CUP I53D23002420006- funded by European Union - Next Generation EU  within the PRIN 2022 program (D.D. 104 - 02/02/2022 Ministero dell'Universit\`{a} e della Ricerca). This manuscript reflects only the authors' views and opinions and the Ministry cannot be considered responsible for them.
\\
Moreover, J. L. V\'azquez was funded by  grant PID2021-127105NB-I00 from MICINN (the Spanish Government). J.~L.~V\'azquez is an Honorary Professor at Univ. Complutense de Madrid.


\

\

\end{document}